\documentclass[12pt]{amsart}
\usepackage{graphicx}
\usepackage{amsmath}
\usepackage{amssymb}
\usepackage{amsthm}
\usepackage{tikz}
\usepackage{tikz-cd}
\usepackage[utf8]{inputenc}
\usepackage{comment}

\newtheorem{theorem}{Theorem}[section]
\newtheorem{lemma}[theorem]{Lemma}

\newtheorem{corollary}[theorem]{Corollary}
\newtheorem{proposition}[theorem]{Proposition}
\newtheorem{observation}[theorem]{Observation}

\theoremstyle{definition}
\newtheorem{definition}[theorem]{Definition}

\newtheorem{example}[theorem]{Example}
\newtheorem{remark}[theorem]{Remark}
\newtheorem{problem}[theorem]{Problem}
\newtheorem{notation}[theorem]{Notation}
\numberwithin{equation}{section}
\numberwithin{equation}{section}

\newcommand{\iLim}{\varprojlim}
\newcommand{\treemon}{\mathcal{T}_\mathcal{M}}
\newcommand{\treet}{\mathcal{T}_{\mathcal{M}\it 3}}
\newcommand{\treecon}{\mathcal{T}_\mathcal{C} }
\newcommand{\treend}{\mathcal{T}_\mathcal{CE} }
\DeclareMathOperator{\diam}{diam}
\DeclareMathOperator{\ord}{ord}

\DeclareMathOperator{\id}{id}
\begin{document}

\title{ Projective Fra\"{\i}ss\'{e} limits of trees}

\author[W. J. Charatonik]{W\l odzimierz J. Charatonik}

\address{W. J. Charatonik and R. P. Roe, Department of Mathematics and Statistics\\
         Missouri University of Science and Technology\\
         400 W 12th St\\
         Rolla MO 65409-0020}
\curraddr{W. J. Charatonik, ul. Orzeszkowej 78 m. 3, 50-311 Wrocław, Poland.}
\email{wjcharat@mst.edu}
\author[R. P. Roe]{Robert P. Roe}
\curraddr{Robert Roe, 3845 Fairway Dr, Florissant, MO 63033, USA}

\email{rroe@mst.edu}

\date{\today}


\begin{abstract}
The following paper has been withdrawn from consideration for publication because there are mistakes.  In particular, Theorem 3.9 does not hold. Examples were found of finite trees with monotone epimorphisms which do not amalgamate. Further, finite rooted trees with monotone epimorphisms do not amalgamate. A revision, with additional co-authors A. Kwiatkowska and S. Yang, is posted on arXiv ({\it Projective  Fra\"{\i}ss\'{e} limits of trees with confluent epimorphisms} 2312.16915). In that article, it is shown that the family of finite trees having ramification vertices of order at most 3 with monotone epimorphisms does form a projective Fra\"{\i}ss\'e family and the topological realization of its Fra\"{\i}ss\'e limit is the Wa\. zewski dendrite $D_3$.  Further, two families of finite rooted trees with restrictions on what confluent epimorphisms are allowed are also shown to form projective Fra\"{\i}ss\'e families. The topological realization of the Fra\"{\i}ss\'e limit of one of these families is shown to be the Mohler-Nikiel universal dendroid.

We continue study of projective Fra\"{\i}ss\'e limit developed by Irvin, Panagiotopoulos and  Solecki. We modify the ideas of monotone, confluent, and light mappings from continuum theory as well as several properties of continua so as to apply to topological graphs. As the topological realizations of the projective Fra\"{\i}ss\'e limits we obtain  the dendrite $D_3$ as well as quite new, interesting continua for which we do not yet have topological characterizations. 
\end{abstract}

\maketitle

\section{Introduction and Definitions}
In \cite{Pseudo}, T. Irwin and S. Solecki introduced the idea of a projective Fra\"{\i}ss\'e limit as a dualization of the injective Fra\"{\i}ss\'e limit from model theory. In that paper they construct the pseudo arc as the topological realization of a projective Fra\"{\i}ss\'e limit of a certain class of finite graphs and epimorphisms between members of the class.  Subsequently, D. Barto\v sov\' a and A. Kwiatkowska, \cite{B-K}, Kubiś and A. Kwiatkowska, \cite {kubis}  and A. Panagiotopoulos and S. Solecki, \cite {Menger}, extended these ideas to study, repectively, the Lelek fan, the Lelek fan and the Poulsen simplex, and the Menger curve as Fra\"{\i}ss\'e limits.  

In this article we continue this study by considering families of finite trees and epimorphims (mappings) between them that satisfy certain properties. We modify the ideas of monotone, confluent, and light mappings from continuum theory as well as several properties of continua so as to apply to topological graphs. Tools are developed for studying projective Fra\"{\i}ss\'e limits and their topological realizations.
As the topological realization of the projective Fra\"{\i}ss\'e limit of finite trees with monotone epimorphisms we obtain the dendrite $D_3$, (see Theorem~\ref{D3}). In the cases where epimorphisms are restricted to being confluent order-preserving or confluent and end-preserving we show that the families are projective Fra\"{\i}ss\'e families (see Theorems~\ref{confluent-projective-family} and \ref{Thm-confluent-end-preserving}) and that the Fra\"{\i}ss\'e limits have admit topological realizations. These continua have many interesting properties (see Theorems~\ref{summerizing-confluent} and \ref{Thm-CE}) and they may be different from any known continua.

We start with some basic definitions and background about projective Fra\"{\i}ss\'e families and limits.

A {\it graph} is an ordered pair $A=( V(A), E(A))$, where $E(A)\subseteq V(A)^2$ is a reflexive and symmetric relation on $V(A)$. The elements of $V(A)$ are called {\it vertices} of graph $A$ and elements of $E(A)$ are called {\it edges}.

A {\it topological graph} $K$ is a graph $( V(K),E(K))$, whose domain $V(K)$ is a 0-dimensional, compact, second-countable (thus metic) space and $E(K)$ is a closed, reflexive and symmetric  subset of $ V(K)^2$. A topological graph is an example of topological $\mathcal{L}$ structure. For a general definition of a topological $\mathcal{L}$ structure see \cite{Pseudo}.

Given two graphs $A$ and $B$ a function $f\colon V(A)\to V(B)$ is a homomorphism if it maps edges to edges, i.e.  $\langle a,b\rangle \in E(A)$ implies $\langle f(a),f(b)\rangle \in E(B)$.  
If we consider rooted trees, then we have an additional relation, namely an order $\le$ on vertices, and 
homomorphisms have to preserve the order. Details are in Section \ref{rooted-trees}.
A homomorphism $f$ is an epimorphism if it is moreover surjective
on both vertices and edges. An isomorphism is an injective epimorphism.

\begin{definition}\label{definition-order}
A finite graph  $T$ is a {\it tree} if  for every two distinct vertices $a,b \in T$ there is a unique finite sequence
$v_0=a,v_1, \dots , v_n=b$ of vertices in $T$ such that for every $i\in \{0,1,\dots n-1\}$ we have $\langle v_i,v_{i+1}\rangle\in E(T)$
and  $v_i\ne v_{i+1}$. 
Let $n$ be a natural number, a vertex $p\in T$ has order  $n$ ($\ord(p)=n$) if there are $n$ non-degenerate edges in $T$ that contain $p$. If $\ord(p)=1$ then $p$ is an end vertex, $\ord(p)=2$ then $p$ is an ordinary vertex and if $\ord(p)\ge 3$ then $p$ a ramification vertex.
\end{definition}
The theory of projective Fra\"{\i}ss\'e limits were developed in \cite{Pseudo} and further refined in \cite{Menger}. We literally recall their definitions here.

\begin{definition}\label{definition-Fraisse}
Let $\mathcal{F}$ be a class of finite graphs with a fixed family of morphisms among
the structures in $\mathcal{F}$. We assume that each morphism is an epimorphism with respect
to $\mathcal{F}$. We say that $\mathcal{F}$ is a projective Fra\"{\i}ss\'e class if
\begin{enumerate}
\item $\mathcal{F}$ is countable up to isomorphism, that is, any sub-collection of pairwise
non-isomorphic structures of $\mathcal{F}$ is countable;
\item morphisms are closed under composition and each identity map is a morphism;
\item for $B,C \in  \mathcal{F}$ there exist $D\in \mathcal{F}$ and morphisms $f\colon  D \to B$ and $g\colon  D \to C$;
and
\item for every two morphisms $f\colon B \to A$ and $g\colon C \to A$, there exist morphisms
$f_0\colon D \to B$ and $g_0\colon D\to C$ such that $f \circ f_0 = g \circ g_0$, i.e. the diagram (D1) commutes.
\end{enumerate}

\begin{equation}\tag{D1}
\begin{tikzcd}
&B\arrow{ld}[swap]{f}\\
A&&D\arrow[lu,swap,dotted,"f_0"] \arrow[ld,dotted,"g_0"]\\
&C\arrow[lu,"g"]
\end{tikzcd}
\end{equation}
We will refer to the last property above as the projective amalgamation property.
\end{definition}

The class $\mathcal F$ of finite graphs and epimorphisms is enlarged to a class $\mathcal F^\omega$ which includes all topological graphs obtained as inverse limits of graphs in $\mathcal F$ with bonding maps from the family of epimorphisms.  If $G=\iLim\{G_n,\alpha_n\} \in \mathcal F^\omega$ and $a=(a_n)$ and $b=(b_n)$ are elements of $G$ then $\langle a,b\rangle$ is an edge in $G$ if and only if for each $n$, $\langle a_n, b_n\rangle$ is an edge in $G_n$. An epimorphism $h$ between a topological graph $G=\iLim\{G_n,\alpha_n\}$ in $\mathcal F^\omega$ and a finite graph $A\in \mathcal F$ is in the class $\mathcal F^\omega$ if and only if there is an $m$ and an epimorphism  $h'\colon G_m \to A$, $h'\in \mathcal F$, such that $h= h'\circ \alpha_m^\infty$ where $\alpha_m^\infty$ is the canonical projection from the inverse limit space onto the $m$th factor space.  Finally, if $K$ and $L$ are inverse limit spaces in $\mathcal F^\omega$ an epimorphism $h\colon L \to K$ is in the family $\mathcal F^\omega$ if and only if for any finite graph $A \in \mathcal F$ and any epimorphism $g\colon K \to A$ in $\mathcal F^\omega$, $g\circ h \in \mathcal F^\omega$.

In the proof of \cite[Theorem 2.4]{Pseudo} an inverse sequence satisfying certain properties, see definition below, was introduced and used to show the existence and uniqueness of the projective Fra\"{\i}ss\'e limit.

\begin{definition}\label{fund-seq-def}

Given a projective Fra\"{\i}ss\'e family $\mathcal F$ an inverse sequence $\{F_n,\alpha_n\}$ where $F_n\in {\mathcal F}$ and $\alpha_n\colon F_{n+1} \to F_n$ are epimorphisms in $\mathcal F$ is said to be a {\it fundamental sequence} for $\mathcal F$ if the following two conditions hold.

\begin{enumerate}
    \item For any $G\in {\mathcal F}$ there is an $n$ and an epimorphism from $F_n$ onto $G$;
    \item For any $n$, any pair $G, H \in {\mathcal F}$, and any epimorphisms $g\colon H \to G$ and $f\colon F_n \to G$ there exists $m >n$ and an epimorphism $h\colon F_m \to H$ such that $g\circ h = f\circ \alpha_n^m$  i.e the diagram (D2) commutes.
\end{enumerate}

\begin{equation}\tag{D2}
\begin{tikzcd}
F_n\arrow{d}[swap]{f}&F_m\arrow[d,dotted,"h"]\arrow{l}[swap]{\alpha^m_n} \\
G&\arrow[l,"g"]H
\end{tikzcd}
\end{equation}

\end{definition}

Note that the name fundamental sequence has not been standardized. Other names that have been used include generic sequence and  Fra\"{\i}ss\'e sequence.

\begin{theorem}\label{limit}
 Let $\mathcal {F}$ be a projective Fra\"{\i}ss\'e class with a fixed family of epimorphisms among the structures of $\mathcal{F}$.  There exists a unique topological graph  $\mathbb{F}$ such that
\begin{enumerate}
    \item for each $A\in \mathcal{F}$, there exists an epimorphism 
    from $\mathbb{F}$ to $A$;
    \item for $A,B \in \mathcal{F}$ and epimorphisms $f\colon \mathbb{F} \to A$ and $g\colon B\to A$
    there exists an epimorphism $h\colon \mathbb{F}\to B$
    such that $f=g\circ h$.
    \item\label{refinement}
For every $\varepsilon>0$  there is a graph $G\in \mathcal F$ and an epimorphism $f\colon \mathbb{F} \to G$
such that $\diam(f^{-1}(x))<\varepsilon $ for each $x$ in $V(G)$.
\end{enumerate}
This topological graph $\mathbb F$ is called the projective Fra\"{\i}ss\'e limit of $\mathcal F$.
\end{theorem}

\begin{proof}
The uniquness of $\mathbb F$ along with the first two conditions are precisely \cite[Theorem 3.1]{Menger}. For the third condition note that in the proof of \cite[Theorem 3.1]{Menger} it is shown that the projective Fra\"{\i}ss\'e limit $\mathbb F$ is the inverse limit of a fundamental sequence of $\mathcal F$.  Condition~(3) then follows from the definition of a metric on the inverse limit space.  
\end{proof}

\begin{definition}
Given a topological graph $K$, if $E(K)$ is also transitive then it is an equivalence relation and $K$ is known as a {\it prespace}. The quotient space $K/E(K)$ is called the {\it topological realization} of $K$ and is denoted by $|K|$.
\end{definition}

\begin{theorem}
Each compact metric space is a topological realization of a topological graph.
\end{theorem}
\begin{proof}
Let $X$ be a compact metric space and let $f\colon C\to X$ be a surjective mapping from the Cantor set $C$. Define a topological graph $K$ by putting $V(K)=C$ and $\langle a,b\rangle\in E(K)$ if and only if $f(a)=f(b)$. Then $K$ is a compact topological graph, $E(K)$ is transitive, and $|K|$ is homeomorphic to $X$.
\end{proof}

\section{Connectedness properties of topological graphs}

In this section we propose definitions of some connectedness properties of topological graphs analogous to
respective definitions for continua. We have decided to keep the terminology original to continuum theory.
Let us start with the definitions of connected and locally connected topological graphs as in \cite{Menger}.

\begin{definition}
  Given a topological graph $G$, a subset $S$ of $V(G)$ is {\it disconnected} if there are two nonempty closed subsets $P$ and $Q$ of $S$ such that $P\cup Q=S$ and if
  $a\in P$ and $b\in Q$, then $\langle a,b\rangle\notin E(G)$. A subset $S$ of $V(G)$ is {\it connected} if it is not disconnected.
  A graph $G$ is connected if $V(G)$ is connected.
\end{definition}

\begin{definition}
A topological graph $G$ is {\it locally connected} if it admits a basis of its topology consisting of connected sets in the above sense.
\end{definition}

\begin{proposition}\label{connected-image}
If $f\colon G\to H$ is an epimorphism between topological graphs and $G$ is connected, then $H$ is connected.
\end{proposition}
\begin{proof}
Suppose $H$ is disconnected. Then there are two disjoint nonempty closed subsets $A$ and $B$ of $V(H)$ such that $V(H)=A\cup B$ and there is no edges between $A$ and $B$.
So $V(G)=f^{-1}(A)\cup f^{-1}(B)$ and $f^{-1}(A)\cap f^{-1}(B)=\emptyset$. Since $G$ is connected, there are vertices $a\in f^{-1}(A)$ and $b\in f^{-1}(B)$ such that
$\langle a,b\rangle\in E(G)$, and thus $\langle f(a),f(b)\rangle\in E(H)$, a contradiction.
\end{proof}

\begin{definition}
Given a graph $G$, a subset $S$ of $V(G)$, and a vertex $a\in S$ the {\it component of $S$ containing $a$} is the
largest connected subset $C$ of $S$ that contains $a$; in other words $C=\bigcup\{P\subseteq S: a\in P\ \text{ and }P \text{ is connected} \} $.
\end{definition}

\begin{definition}
 A topological graph $G$ is called {\it hereditarily unicoherent} if for every two closed connected subgraphs $P$ and $Q$ of $G$ the
 intersection $P\cap Q$ is connected. The graph $G$ is {\it unicoherent} if in addition $P \cup Q = G$.
\end{definition}

Notice that if  a finite graph is unicoherent then it is hereditarily unicoherent.

\begin{notation}
If $G$ is a topological graph and $S\subseteq V(G)$ by $G\backslash S$ we mean a topological graph $F$ such that $V(F)=V(G)\backslash S$ and if $e=\langle a,b\rangle \in E(G)$ then $e\in F$ if and only if $\{a,b\}\subseteq V(F)$
\end{notation}

The following theorem is known in continuum theory.  Here we give the analogous theorem for topological graphs together with the proof.

\begin{proposition}\label{intersection-hu}
If $G$ is a hereditarily unicoherent topological graph and $\{G_\alpha: \alpha\in A\}$ is a family of connected subgraphs of $G$, then the intersection
$\bigcap\{G_\alpha: \alpha\in A\}$ is connected.
\end{proposition}
\begin{proof}
Suppose $J=\bigcap\{G_\alpha: \alpha\in A\}$ is disconnected. Then there are two disjoint closed subsets $H$ and $K$ of $V(J)$
such that $H\cap K=\emptyset$ and $H\cup K=V(J)$.
Let $H^*$ and $K^*$ be disjoint open sets containing $H$ and $K$ repsectively. The collection $\{G\setminus G_\alpha : \alpha \in A\} \cup (H^* \cup K^*)$ is an open cover of $G$. By compactness, there is a finite open subcover $G\setminus G_{\alpha_1}, \dots, G\setminus G_{\alpha_n}, H^* \cup K^*$ of $G$.
Taking the complements we have that
$H \cup K \subseteq G_{\alpha_1}\cap\dots\cap G_{\alpha_n} \subseteq H^* \cup K^*$ and $G_{\alpha_1}\cap\dots\cap G_{\alpha_n}$ is, by hereditary unicoherence, a connected graph, a contradiction.

\end{proof}

\begin{definition}\label{arc-def}
  We say that a topological graph $G$ is an {\it arc} if it is connected and there are two vertices $a,b\in V(G)$ such that for every
  $x\in V(G)\setminus \{a,b\}$ the graph $G\setminus \{x\}$ is not connected. The vertices $a$ and $b$ are called end vertices of the arc and we say
  that $G$ joins $a$ and $b$.
\end{definition}

In topology there is only one, up to homeomorphism, metric arc, while for topological graphs we have finite arcs, countable arcs as well as uncountable arcs.

\begin{example}
  Let $G$ be the topological graph whose set of vertices is the Cantor ternary set, and the edges are defined by $\langle a,b\rangle\in E(G)$
  if and only if $a=b$ or $a$ and $b$ are end vertices of the same deleted interval. Then $G$  is an uncountable arc, whose topological realization is homeomorphic to $[0,1]$.
\end{example}

\begin{definition}
  A topological graph $G$ is called {\it arcwise connected} if for every two vertices $a,b\in V(G)$ there is a subgraph of $G$ that is an arc and contains $a$ and $b$.
\end{definition}

\begin{definition}\label{dendroid-def}
  A hereditarily unicoherent and arcwise connected finite graph is called a {\it dendroid}. A locally connected dendroid is called a {\it dendrite}.
\end{definition}

\begin{proposition}
  For a connected finite graph $G$ the following conditions are equivalent:
  \begin{enumerate}
    \item $G$ is a tree;
    \item  $G$ is unicoherent;
    \item $G$ is a dendroid;
    \item $G$ is a dendrite;
  \end{enumerate}
\end{proposition}
\begin{proof}
  The implications $(1)\implies(2)\implies(3)\implies(4)$ follow from definitions. To see $(4)\implies(1)$ note that a finite dendrite cannot contain a cycle, so it is a tree.
\end{proof}

The following observations follow from the fact that a topological graph $G$
is connected if and only if its topological realization $|G|$
is connected.

\begin{observation}
If $G$ is a hereditarily unicoherent topological graph and $E(G)$ is transitive, then its topological realization $|G|$ is a hereditarily unicoherent continuum.
\end{observation}

\begin{observation}\label{arcs}
If a topological graph $G$ is an arc and $E(G)$ is transitive, then its topological realization $|G|$ is an arc or a point.
\end{observation}

\begin{observation}
If a topological graph $G$ is arcwise connected and $E(G)$ is transitive, then its topological realization $|G|$ is arcwise connected.
\end{observation}

\begin{observation}\label{realization-dendroid}
If a topological graph $G$ is a dendroid and $E(G)$ is transitive, then its topological realization $|G|$ is a dendroid.
\end{observation}

\begin{observation}\label{top-dendrite}
If a topological graph $G$ is a dendrite and $E(G)$ is transitive, then its topological realization $|G|$ is a dendrite.
\end{observation}

The following Theorem provides a sufficient condition in order for a projective Fra\"{\i}ss\'e family  to have transitive set of edges in the projective Fra\"{\i}ss\'e limit.

\begin{theorem}\label{transitive}
Suppose that $\mathcal G$ is a projective Fra\"{\i}ss\'e family of graphs and for every $G\in \mathcal G$, for every $a,b,c\in V(G)$ such that
$\langle a,b\rangle\in E(G)$ and $\langle b,c\rangle\in E(G)$ there is a graph $H$ and an epimorphism $f^H_G\colon  H\to G$ such that for every vertices $p,q,r\in V(H)$ such that $f^H_G(p)=a$, $f^H_G(q)=b$, and $f^H_G(r)=c$ we have $\langle p,q\rangle\notin E(H)$ or $\langle q,r\rangle\notin E(H)$. Then there are not distinct vertices $a,b,c$ in the Fra\"{\i}ss\'e limit $\mathbb G$ such that the edges $\langle a,b\rangle$ and $\langle b,c \rangle$ are in $\mathbb G$ hence $\mathcal G$ has a transitive set of edges.
\end{theorem}
\begin{proof}
Let $\mathcal G$ be a family that satisfies the assumptions of the Theorem.
Suppose on the contrary that there are three vertices $a,b,c\in \mathbb G$ such that $\langle a,b\rangle, \langle b,c\rangle\in E(\mathbb G)$. Let a graph $G\in\mathcal G$ and an epimorphism $f_G\colon \mathbb G\to G$ be such that $f_G(a), f_G(b), \text{  and }f_G(c)$ are three distinct vertices of $G$. Then $\langle f_G(a),f_G(b)\rangle, \langle f_G(b),f_G(c)\rangle\in E(G)$, and thus, by our assumption, there is a graph
$H$ and an epimorphism $f^H_G\colon  H\to G$ such that for an epimorphism $f_H\colon \mathbb G\to H$ satisfying $f_G=f^H_G\circ f_H$ we have  $\langle f_H(a),f_H(b)\rangle\notin E(H)$ or $\langle f_H(b),f_H(c)\rangle\notin E(H)$. This contradicts the fact that $f_H$ maps edges to edges.
\end{proof}

\begin{theorem}\label{limit-of-hu}
If $\mathcal T$ is a projective Fra\"{\i}ss\'e family of  trees, and $\mathbb T$ is a projective Fra\"{\i}ss\'e
limit of $\mathcal T$, then $\mathbb T$ is hereditarily unicoherent.
\end{theorem}
\begin{proof}
Suppose $\mathbb T$ is not hereditarily unicoherent.  Then there exist closed  connected subsets $P$ and $Q$ of $V(\mathbb T)$ such that $P \cap Q$ is not connected. Let $p \in P\backslash Q$ and $q \in Q\backslash P$.  By Condition \ref{refinement} of Theorem \ref{limit}, choosing $\varepsilon$ small enough, there exists a tree $G$ and a epimorphism $f_G\colon  \mathbb T \to G $ such that $f_G(p) \not \in f_G(Q)$,  $f_G(q)\not \in f_G(P)$, and $f_G(P) \cap f_G(Q)$ is not connected.    This contradicts the fact that $G$ is a tree.
\end{proof}

\section{Monotone epimorphisms}

In this section we investigate that the family of finite trees with monotone epimorphisms. We show that this family, which we denote as $\mathcal T_{\mathcal M}$, is a projective Fra\"{\i}ss\'e family  (Theorem~\ref{amalgamation-monotone}) and that the topolotical realization of $\mathcal T_{\mathcal M}$ is homeomorphic to $D_3$, the standard universal denrite of order 3, also known as a Wa\. zewski denrite of order 3  (Theorem~\ref{D3}). Note that Kwiatkowska \cite{K-D3} uses an injective Fra\"{\i}ss\'e construction of $D_3$ to study the group of homeomorphisms of Wa\. zewski denrites.

First, we want to show an example that usually epimorphisms between trees do not have amalgamations that are connected graphs.
That is why we need to restrict considered classes of epimorphisms.

\begin{example}

There exist a triod $T$, arcs $I$ and $J$ for which there is no connected graph $G$ and epimorphisms $f_0$ and $g_0$ such that the diagram below commutes.

\begin{equation}\tag{D2}
\begin{tikzcd}
&I\arrow{ld}[swap]{f}\\
T&&G\arrow[lu,swap,dotted,"f_0"] \arrow[ld,dotted,"g_0"]\\
&J\arrow[lu,"g"]
\end{tikzcd}
\end{equation}

The triod $T$ has the center $b$ and the end vertices $a$, $c$, and $d$; the arc $I$ has vertices $p_1,p_2,p_3,p_4$, and $p_5$, similarly
the arc  $J$ has vertices $q_1,q_2,q_3,q_4$, and $q_5$. The epimorphisms $f$ and $g$ are pictured below. Precisely, we have $f(p_1)=d$,
$f(p_2)=b$, $f(p_3)=c$, $f(p_4)=b$, and $f(p_5)=a$; similarly,  $g(q_1)=d$,
$g(q_2)=b$, $g(q_3)=a$, $g(q_4)=b$, and $g(q_5)=c$.

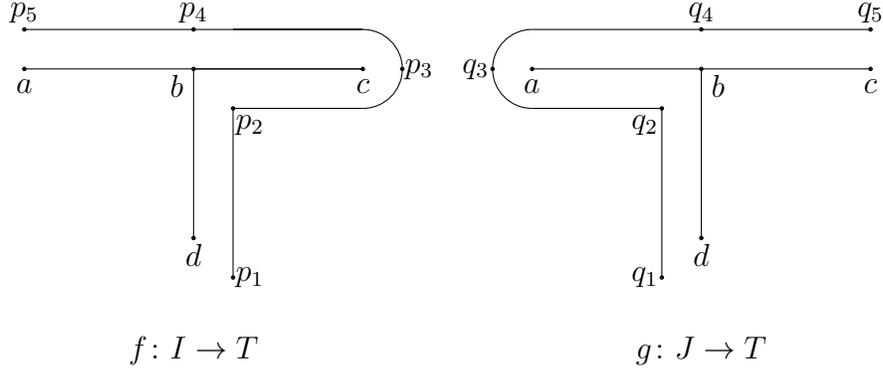
\begin{center}
\begin{tikzpicture}[scale=0.75]

  \draw (3,0) -- (3,3);
   \draw (3.7,-0.7) -- (3.7,2.3) -- (6,2.3);
   \draw (3.7,3.7) -- (6,3.7) -- (0,3.7);
  \draw (0,3) -- (3,3);
  \draw (6,3) -- (3,3);
\draw (6,2.3) arc (-90:90:0.7);

  \node at (3,-0.3) {$d$};

  \draw (6,3) -- (3,3);
  \node at (0,2.7) {$a$};
\node at (2.7,2.7) {$b$};
\node at (6,2.7) {$c$};
\node at (4,-0.7) {$p_1$};
\node at (4,2) {$p_2$};
\node at (7,3) {$p_3$};
\node at (3,4) {$p_4$};
\node at (0,4) {$p_5$};
\draw (3,3) circle (0.03);
\draw (3,3) circle (0.015);

\draw (3,0) circle (0.03);
\draw (3,0) circle (0.015);

\draw (0,3) circle (0.03);
\draw (0,3) circle (0.015);

\draw (6,3) circle (0.03);
\draw (6,3) circle (0.015);

\draw (3,3.7) circle (0.03);
\draw (3,3.7) circle (0.015);

\draw (0,3.7) circle (0.03);
\draw (0,3.7) circle (0.015);

\draw (6.7,3) circle (0.03);
\draw (6.7,3) circle (0.015);

\draw (3.7,2.3) circle (0.03);
\draw (3.7,2.3) circle (0.015);

\draw (3.7,-0.7) circle (0.03);
\draw (3.7,-0.7) circle (0.015);

  \draw (12,0) -- (12,3);
    \draw (9,3) -- (12,3);
  \draw (15,3) -- (12,3);
    \draw (9,3.7) -- (15,3.7);

    \node at (12,-0.3) {$d$};

  \node at (9,2.7) {$a$};
\node at (12.3,2.7) {$b$};
\node at (15,2.7) {$c$};

\draw (9,3.7) arc (90:270:0.7);

  \draw (9,2.3) -- (11.3,2.3) -- (11.3,-0.7);

  \node at (11,-0.7) {$q_1$};
\node at (11,2) {$q_2$};
\node at (8,3) {$q_3$};
\node at (12,4) {$q_4$};
\node at (15,4) {$q_5$};

\draw (12,3) circle (0.03);
\draw (12,3) circle (0.015);

\draw (12,0) circle (0.03);
\draw (12,0) circle (0.015);

\draw (9,3) circle (0.03);
\draw (9,3) circle (0.015);

\draw (15,3) circle (0.03);
\draw (15,3) circle (0.015);

\draw (15,3.7) circle (0.03);
\draw (15,3.7) circle (0.015);

\draw (12,3.7) circle (0.03);
\draw (12,3.7) circle (0.015);

\draw (8.3,3) circle (0.03);
\draw (8.3,3) circle (0.015);

\draw (11.3,2.3) circle (0.03);
\draw (11.3,2.3) circle (0.015);

\draw (11.3,-0.7) circle (0.03);
\draw (11.3,-0.7) circle (0.015);

\node at (3,-2) {$f\colon I\to T$};
\node at (12,-2) {$g\colon J\to T$};

\end{tikzpicture}
\end{center}

Suppose that there is a connected graph $G$ and epimorphisms $f_0\colon G\to I$ and $g_0\colon G\to J$ such that the diagram (D2) commutes. Let $x_0\in (f_0)^{-1}(p_1)$ and let
$x_0,x_1,\dots x_n$ be a sequence of vertices of $G$ such that:
\begin{enumerate}
    \item for each $i \in \{0,1,\dots, n-1\}$ we have $\langle x_i,x_{i+1}\rangle\in E(G)$;
    \item $x_0,x_1,\dots x_{n-1}\in (f\circ f_0)^{-1}(\{b,d\})$;
    \item $x_{n}\notin (f\circ f_0)^{-1}(\{b,d\})$.
\end{enumerate}
Then we have
\begin{enumerate}
    \item $f_0(x_0)=p_1$,
    \item $f_0(x_0), f_0(x_1), \dots f_0(x_{n-1})\in \{p_1,p_2\}$;
    \item $f_0(x_n)=p_3$;
    \item $g_0(x_0)=q_1$,
    \item $g_0(x_0), g_0(x_1), \dots g_0(x_{n-1})\in \{q_1,q_2\}$;
    \item $g_0(x_n)=q_3$.
    \end{enumerate}
By conditions (3) and (6) we have that $f(f_0(x_n))=c$, while $g(g_0(x_n))=a$, so the diagram (D2) does not commute, a contradiction.

\end{example}

In this section we consider the category $\treemon$ of trees with morphisms being monotone epimorphisms. We start with necessary definitions.

\begin{definition}
Given two topological graphs $G$ and $H$ an epimorphism $f\colon G\to H$  is called {\it monotone} if the preimage of a connected
subset of $V(H)$ is a connected subset of $V(G)$, or, equivalently (see \cite [Lemma 1.1] {Pseudo}) if the preimage of every
vertex in $V(H)$ is connected.
\end{definition}
The following example shows that the concept of monotone epimorphism for graphs is not exactly how continuum theory sees monotone maps between continua.
\begin{example}
 There is a monotone epimorphism from a cyclic graph onto an arc.

 Let $G$ be a complete graph with three vertices $a,b$, and $c$, i.e. a graph where
 $E(G)=V(G)^2$, and let $H$ be an arc with two vertices $p$ and $q$, and $E(H)=V(H)^2$. Define $f\colon G\to H$ by $f(a)=p$ and $f(b)=f(c)=q$. The reader can
 verify that $f$ is a monotone epimorphism.
\end{example}

The following result, which is well known in continuum theory, also holds in the setting of topological graphs and is used in Theorem~\ref{D3}.

\begin{lemma}\label{images-of-arcs}
If $f\colon G\to H$ is a monotone epimorphism between topological graphs and $G$ is an arc, then $H$ is an arc and the images of end vertices of $G$ are end vertices of  $H$.
\end{lemma}

\begin{proof}
Denote the end vertices of $G$ by $a$ and $b$. We need to show that every vertex in $V(H)\setminus \{f(a),f(b)\}$ disconnects $H$. Let $y\in V(H)\setminus \{f(a),f(b)\}$; then, since
$G$ is an arc the graph $G\setminus f^{-1}(y)$ is disconnected. Let $G\setminus f^{-1}(y)$ be the union of two disjoint graphs $G\setminus f^{-1}(y)=U\cup V$. Thus $V(H)\setminus \{y\}=f(U)\cup f(V)$, so it is disconnected as needed.
\end{proof}

\begin{proposition}\label{monotone-projections}
If $\mathcal G$ is a projective Fra\"{\i}ss\'e family of graphs with monotone epimorphisms and $\mathbb G$ is a projective Fra\"{\i}ss\'e 
limit of $\mathcal G$, then for every $G \in \mathcal G$ any epimorphism $f_G\colon \mathbb G\to G$ is monotone.
\end{proposition}
\begin{proof}
Suppose the contrary; then there is a graph $G\in \mathcal G$, a vertex $a\in G$ and an epimorphism $f_G\colon \mathbb G\to G$ such that
$f_G^{-1}(a)$ is the disjoint union of two nonempty closed subsets $A$ and $B$. Choose $\varepsilon>0$ such that $d(a,b)>\varepsilon$
for every $a\in A$ and $b\in B$. By conditions (2) and (3) of Theorem \ref{limit} and amalgamation we may obtain a graph $H\in \mathcal G$, an epimorphism $f^H_G\colon H\to G$, and an epimorphism
$f_H\colon \mathbb G\to H$ be such that  $f_G=f^H_G\circ f_H $ and $f_H$ is an $\varepsilon$ epimorphism.  Then $f_H(A)$ and $f_H(B)$ are two disjoint nonempty subsets of $H$ and $(f_G^H)^{-1}(a)=f_H(A)\cup f_H(B)$ contrary to monotonicity of $f_G^H$.
\end{proof}

\begin{proposition}\label{monotone-limit-of-arcs}
If $\mathcal T$ is a projective Fra\"{\i}ss\'e family of finite arcs with monotone epimorphisms, and $\mathbb T$ is a projective Fra\"{\i}ss\'e 
limit of $\mathcal T$, then $\mathbb T$ is an arc.
\end{proposition}

\begin{proof}

First we claim there are two vertices $a$ and $b$ in $\mathbb T$ such that for any $T\in \mathcal {T}$ and any monotone epimorphism $f_T\colon \mathbb T \to T$, $f_T(a)$ and $f_T(b)$ are end vertices of $T$. Let $\{T_n,\alpha_n\}$ be a fundamental sequence for $\mathcal T$. The monotone epimorphisms $\alpha_n$ take end vertices to end vertices so there are exactly two vertices $a=(a_n)$ and $b=(b_n)$ in $\mathbb T$ such that for each $n$, $a_n$, and $b_n$, is an end vertex for $T_n$. For  $T \in \mathcal T$ and $f_T\colon \mathbb T \to T$ there is an $n$ and  monotone epimorphisms $g\colon T_n \to T$ and $f_{T_n}\colon \mathbb T \to T_n$ such that $f_T = g \circ f_{T_n} $. Thus $f_T$ maps the vertices $a$ and $b$ to end vertices of $T$ as claimed.

We need to show that any vertex $x\in\mathbb T\setminus \{a,b\}$ disconnects $\mathbb T$. For an arc $T\in \mathcal T$ and a given monotone epimorphism  $f_T\colon \mathbb T\to T$  define $U_{f_T}$ and $V_{f_T}$ as components of $T\setminus \{f_T(x)\}$ containing the vertices $f_T(a)$ and $f_T(b)$ respectively.
Note that for $T,S\in \mathcal{T}$ and a monotone epimorphism $f^{T}_{S}\colon T\to S$ such that $f_S=f^{T}_{S} \circ f_T$ we have $(f^{T}_{S})^{-1}(U_{f_S})\subseteq U_{f_T}$, and thus $f_{S}^{-1}(U_{f_S}) \subseteq f_{T}^{-1}(U_{f_T})$. Letting $U=\bigcup\{U_{f_T}:f_T \text{ is an epimorphism between trees in }\mathcal{T}\}$ and $V=\bigcup\{V_{f_T}:f_T \text{ is an epimorphism between trees in }\mathcal{T}\}$ we have $\mathbb{T}\setminus \{x\}=U\cup V$ and
$U\cap V=\emptyset$, so $\mathbb{T}\setminus \{x\}$ is not connected as required.
\end{proof}

\begin{proposition}\label{arwise-connected-limits}
If $\mathcal T$ is a projective Fra\"{\i}ss\'e family of finite trees with monotone epimorphisms, and $\mathbb T$ is a projective Fra\"{\i}ss\'e 
limit of $\mathcal T$, then $\mathbb T$ is arcwise connected.
\end{proposition}

\begin{proof}
Let $a,b$ be two vertices of $\mathbb T$. We will construct an arc joining $a$ and $b$. Let us use notations as is the proof of Proposition~\ref{monotone-limit-of-arcs}: that is, for trees $T,T_1,T_2 \in \mathcal T$, $f_T$ is a monotone epimorphism that maps the limit $\mathbb T$ onto the tree $T$, while $f^{T_1}_{T_2}$ is a monotone epimorphism between trees in $\mathcal T$.  For a given $T\in \mathcal T$, let $J_T$ be the arc between $f_T(a)$ and $f_T(b)$ and observe that $f^{T_1}_{T_2}(J_{T_1})\subseteq J_{T_2}$, and thus $\{J_T:T\in\mathcal T\}$ is a projective Fra\"{\i}ss\'e family, whose limit $\mathbb J$ is, by Proposition \ref{monotone-limit-of-arcs}, an arc joining $a$ and $b$.
\end{proof}

\begin{corollary}\label{dendrites}
If $\mathcal T$ is a projective Fra\"{\i}ss\'e family of trees with monotone epimorphisms, and $\mathbb T$ is a projective Fra\"{\i}ss\'e 
limit of $\mathcal T$, then $\mathbb T$ is a dendrite.
\end{corollary}

\begin{proof}
The limit $\mathbb T$ is arcwise connected by Proposition~\ref{arwise-connected-limits}, it is hereditarily unicoherent by Theorem \ref{limit-of-hu}, and it is locally connected by \cite[Theorem 2.1]{Menger}.
\end{proof}

\begin{theorem}\label{amalgamation-monotone}
The category $\treemon$ is a projective Fra\"{\i}ss\'e family.
\end{theorem}

\begin{proof}
Conditions (1) and (2) are obvious.

For condition (4) we will use induction on the sum of the number of edges in $B$ and $C$.  If $B$ and $C$ each have just one degenerate edge, that is, $B$ and $C$ are singletons, then $A$ has only one vertex and we can take $D$ to be just a single vertex. 

Now assume there is a natural number $N$ and such that for any trees $A$, $B$, and $C$ with the sum of the number of edges in $B$ and  $C$ is less than or equal to $N$ and $f\colon B\to A$ and $g\colon C\to A$ are monotone epimorphisms then there exists a tree $D$ and monotone epimorphisms $\alpha \colon D\to B$ and $\beta \colon D\to C$ such that $f\circ \alpha = g\circ \beta$.

Suppose $A$, $B$ and $C$ are  trees such that the sum of the number of edges in $B$ and $C$ is $N+1$ and there are monotone epimorphisms $f\colon B\to A$ and $g\colon C\to A$. Let $\langle a, b\rangle$ be an edge in $B$ such that $b$ is an end vertex of $B$.  Let $B'$ be the finite subtree of $B$ without the edge $\langle a,b\rangle$ (still containing the vertex $a$).  There are two possibilities, $f|_{B'}$ is surjective or it is not.

Consider the case that $f|_{B'}$ is surjective.  By the inductive hypothesis there exist a tree $D'$ and monotone epimorphisms $\alpha'\colon  D'\to B'$ and $\beta'\colon  D'\to C$ such that $f|_{B'}\circ \alpha' = g\circ \beta'$ (see the diagram below).

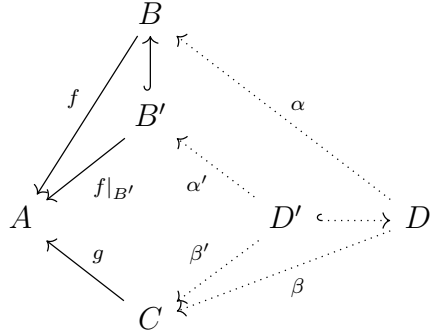
\begin{equation*}
\begin{tikzcd}
&B\arrow{ldd}[swap]{f}\\
&B'\arrow{ld}{f|_{B'}}\arrow[hook,u]\\
A&&D'\arrow[lu,dotted,"\alpha'"] \arrow[ld,dotted,swap,"\beta'"]\arrow[r,dotted,hook]&
D\arrow[lluu,dotted,swap,"\alpha"]\arrow[lld,dotted,"\beta"]\\
&C\arrow[lu,swap,"g"]
\end{tikzcd}
\end{equation*}

Note that $f(b)=f(a)$ otherwise if $f(b) \not = f(a)$ then there exist $c\in B'$ such that $f(c)=f(b)$.  But then $\{b,c\}\in f^{-1}(f(b))$ and $a\not \in f^{-1}(f(b))$ so $f^{-1}(f(b))$ is not connected, contradicting $f$ being monotone.

Chose $c\in (\alpha')^{-1}(a)$ and let $D=D'\cup \langle c,d\rangle$ where $d$ is a vertex not in $D'$.  Let $\alpha\colon D\to B$ be the extension of $\alpha'$ with $\alpha(d) = b$ and $\beta$ be the extension of $\beta'$ with $\beta(d) = \beta'(c)$. Then $D$ is the desired amalgamation.

Now consider the case that $f|_{B'}=A'$, a proper subtree of $A$.  So $E(A)=E(A') \cup\{\langle f(a), f(b)\rangle\}$.  Let $C'=g^{-1}(A')$ and $F=g^{-1}(\{f(b)\})$.  Then $A'$, $B'$ and $C'$ together with $f|_{B'}\colon B' \to A'$ and $g|_{C'}\colon  C'\to A'$ satisfy the inductive hypothesis.  So there exist a  tree $D'$ and monotone epimorphisms $\alpha'\colon D'\to B'$ and $\beta'\colon D' \to C'$ such that $f|_{B'}\circ \alpha' = g|_{C'}\circ \beta'$ (see the diagram below).  

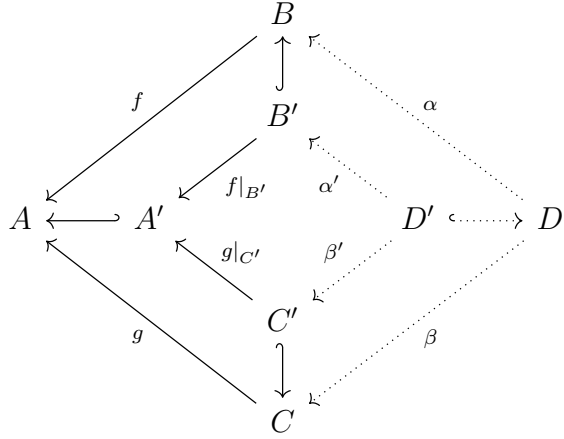
\begin{equation*}
\begin{tikzcd}
&&B\arrow{lldd}[swap]{f}\\
&&B'\arrow{ld}{f|_{B'}}\arrow[hook,u]\\
A&A'\arrow[l,hook']&&D'\arrow[lu,dotted,"\alpha'"] \arrow[ld,dotted,swap,"\beta'"]\arrow[r,dotted,hook]&
D\arrow[lluu,dotted,swap,"\alpha"]\arrow[lldd,dotted,"\beta"]\\
&&C'\arrow[lu,swap,"g|_{C'}"]\arrow[d,hook']\\
&&C\arrow[lluu,"g"]
\end{tikzcd}
\end{equation*}

Let $\langle p,q\rangle$ be an edge in $C$ with $p\in g'^{-1}(f(a))$ and $q\in F$.  Next let $F'$ be an isomorphic copy of $F$ and $D= D' \cup F' \cup \langle p', q'\rangle $ where $p' \in \beta'^{-1}(p)$ and $q'$ is the image of $q$ under the isomorphism.  If $\alpha$ is the extension of $\alpha'$ with $\alpha(c)=b$ for all $c \in F'$ and $\beta$ is the extension of $\beta'$ such that for all $x \in F'$, $\beta(x)$ is the image of $x$ under the isomorphism between $F$ and $F'$, then the tree $D$ together with $\alpha$ and $\beta$ is the desired amalgamation.

Finally, to prove condition (3) we may let $A$ be the tree  consisting of a single vertex and $f$ and $g$ be constant maps, then the tree $D$ and maps $\alpha$ and $\beta$ from the previous part satisfy condition (3).
\end{proof}

Recall (see \cite[(6), p. 490]{JJC-Universal}) that the Wa\. zewski denrite $D_3$ can be characterized by the following conditions:
\begin{enumerate}
    \item each ramification point is of order 3;
    \item the set of ramification points is dense.
\end{enumerate}

We want to show that the topological realization of the projective Fra\"{\i}ss\'e limit of the family $\treemon$ is homeomorphic to the dendrite $D_3$. We divide the proof into two steps.

\begin{theorem}\label{dense}
The projective Fra\"{\i}ss\'e limit of $\treemon$ has transitive set of edges and the topological realization of the Fra\"{\i}ss\'e limit of $\treemon$ is a dendrite with a dense set of ramification points.
\end{theorem}
\begin{proof}

 First we show that the family $\mathcal{T_M}$ satisfies the hypothesis of Theorem \ref{transitive} and thus has a transitive set of edges. Let $G\in \mathcal{T_C}$ and $a,b,c\in V(G)$ such that $\langle a,b\rangle, \langle b,c\rangle \in E(G)$. Define $H\in \mathcal{T_C}$ by $V(H)=V(G)\cup \{a',b'\}$, $E(H)=E(G)\cup\{\langle a,a'\rangle, \langle a',b'\rangle, \langle b',b\rangle\}\setminus\{\langle a,b\rangle\}$, and the epimorphism $f^{H}_G$ where $f^{H}_G(a')=a$, $f^{H}_G(b')=b$ and $f^{H}_G$ is the identity otherwise.  Then, for  $p\in\{a,a'\}$, $q\in \{b,b'\}$, and $r=c$ we have either $\langle p,q\rangle\not \in E(H) $ or $\langle q,r\rangle \not\in E(H)$.

The projective Fra\"{\i}ss\'e limit
$\mathbb D$ of $\treemon$ is a dendrite by Corollary \ref{dendrites}, so its topological realization $|\mathbb D|$ is a (topological) dendrite by Corollary \ref{top-dendrite}. It remains to show that the set of ramification points of $|\mathbb D|$ is dense.

Let us introduce the necessary notation. Let $\varphi\colon \mathbb D\to |\mathbb D|$ be the quotient mapping. To prove the density of the set of ramification points of $|\mathbb D|$ suppose on the contrary that $U$ is an open connected subset of $|\mathbb D|$ that contains no ramification point, i.e $U$ is homeomorphic to $(0,1)$. Then $\varphi^{-1}(U)$
is an open subset of $\mathbb D$. By Property (3) of Theorem \ref{limit}, there is a tree $G$, a monotone epimorphism $f_G\colon \mathbb D\to G$, and $a\in V(G)$ such that
$f_G^{-1}(a)\subseteq \varphi^{-1}(U)$. Define a graph $H$ by putting $V(H)=V(G)\cup \{b,c,d\}$ and $E(H)=E(G)\cup \{\langle a,b\rangle,
\langle b,c\rangle, \langle b,d\rangle\}$. Finally, define the monotone epimorphism $f^H_G\colon H\to G$ by
$$
f^H_G(x)=\begin{cases}  x \text{ if } x\in G\\
a \text{ if } x\in \{b,c,d\}.
\end{cases}
$$
Let $f_H\colon \mathbb D\to H$ be an epimorphism such that $f_G=f^H_G\circ f_H$; note that by Proposition~\ref{monotone-projections} the epimorphisms $f_G$ and $f_H$ are monotone.
Thus, the three sets $A=\varphi(f_H^{-1}(\{a,b\}))$,
 $B=\varphi(f_H^{-1}(\{b,c\}))$, and  $C=\varphi(f_H^{-1}(\{b,d\}))$ are continua that satisfy $ A\cup B\cup C\subseteq U$, $A\not\subseteq B\cup C$, $B\not\subseteq A\cup C$, $C\not\subseteq A\cup B$, and $A \cap B \cap C \not = \emptyset$. This contradicts the fact that $U$ is homeomorphic to $(0,1)$.

\end{proof}

\begin{definition}
Denote by $\treet$ the subfamily of $\treemon$ such that each member of $\treet$ satisfies:
\begin{itemize}
    \item each vertex is of order at most 3;
    \item no two vertices of order 3 are connected by an edge.
\end{itemize}
\end{definition}

\begin{proposition}
The family $\treet$ is cofinal in $\treemon$ and thus the projective Fra\"{\i}ss\'e limits of $\treet$ and $\treemon$ are isomorphic.
\end{proposition}

\begin{proof}
It is enough to construct, for any graph $G$ in $\treemon$, a graph $H$ in $\treet$ and a monotone epimorphism $f\colon H\to G$. The idea is pictured below: we replace each ramification vertex with a set of vertices of orders 2 and 3 as pictured. The epimorphism $f$ shrinks the subgraph of $H$ pictured by dotted lines to the central vertex of $G$.
The details are left to the reader.
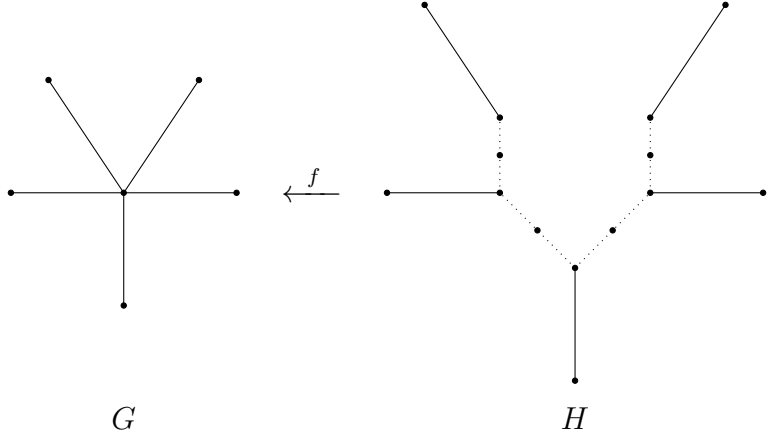
\begin{center}
\begin{tikzpicture}[scale=0.5]

  \draw (2,0) -- (2,3);
  \draw (-1,3) -- (2,3);
  \draw (5,3) -- (2,3);
  \draw (0,6) -- (2,3);
  \draw (4,6) -- (2,3);
 \node at (7,3.25) {$\xleftarrow{\text{\ \ }f \text{\ \  }}$};
 \node at (2,-3) {${G}$};
 \node at (14,-3) {$H$};
\filldraw[black] (2,3) circle (2pt);
\filldraw[black] (2,0) circle (2pt) ;
\filldraw[black] (-1,3) circle (2pt);
\filldraw[black] (5,3) circle (2pt);
\filldraw[black] (0,6) circle (2pt);
\filldraw[black] (4,6) circle (2pt);

\draw (14,-2) -- (14,1);
\draw [dotted] (14,1) -- (16,3);
  \draw (19,3) -- (16,3);
  \draw (9,3) -- (12,3);
  \draw [dotted] (12,3) -- (12,5);
\draw [dotted] (16,3) -- (16,5);
  \draw (16,5) -- (18,8);
  \draw (12,5) --(10,8);
  \draw [dotted] (14,1) -- (12,3);

 \filldraw[black] (14,-2) circle (2pt);
\filldraw[black] (14,1) circle (2pt);
 \filldraw[black] (19,3) circle (2pt);
\filldraw[black] (16,3) circle (2pt);
\filldraw[black] (18,8) circle (2pt);
\filldraw[black] (12,5) circle (2pt);
\filldraw[black] (12,3) circle (2pt);
\filldraw[black] (10,8) circle (2pt);
\filldraw[black] (16,5) circle (2pt);
\filldraw[black] (9,3) circle (2pt);
\filldraw[black] (13,2) circle (2pt);
\filldraw[black] (15,2) circle (2pt);
\filldraw[black] (16,4) circle (2pt);
\filldraw[black] (12,4) circle (2pt);
\end{tikzpicture}
\end{center}

\end{proof}

\begin{theorem}\label{D3}
 The topological realization of the projective Fra\"{\i}ss\'e limit of $\treemon$ is homeomorphic to the  dendrite $D_3$ .
\end{theorem}

\begin{proof}
By Theorem \ref{dense} the topological realization $|\mathbb D| $ of the projective Fra\"{\i}ss\'e limit of $\treemon$ is a dendrite with a dense set of ramification points; it remains to show that every ramification point of $|\mathbb D| $ has order three. Suppose on the contrary that $p\in |\mathbb D|$ is a point of order four or more. Then there are four arcs $A, B, C, D$ in $|\mathbb D|$ such that the intersection any two of them is $\{p\}$.

Note that the transitivity of the edge relation on $\mathbb D$ implies that $\varphi^{-1}(p)$ is a complete graph; on the other hand, by hereditary unicoherence of $\mathbb D$ the set $\varphi^{-1}(p)$ has at most two points. Denote $\varphi^{-1}(p)=\{p_1,p_2\}$ with a possibility that $p_1=p_2$ and note that the sets $\varphi^{-1}(A)$, $\varphi^{-1}(B)$, $\varphi^{-1}(C)$, and $\varphi^{-1}(D)$ contains arcs $\alpha, \beta, \gamma, \delta$ in $\mathbb D$ and that the intersection of any two of these arcs is $\{p_1,p_2\}$.

Let $G$ be a tree and $f_G\colon \mathbb D\to G$ be a monotone epimorphism such that:
\begin{enumerate}
    \item $f_G(\alpha)$, $f_G(\beta)$, $f_G(\gamma)$, and $f_G(\delta)$ are nondegenerate subsets of $V(G)$;
    \item $G\in \treet$;
\end{enumerate}
By Lemma \ref{images-of-arcs} the sets  $f_G(\alpha)$, $f_G(\beta)$, $f_G(\gamma)$, and $f_G(\delta)$ are arcs.
We claim that the intersection of any two of them is $\{f_G(p_1),f_G(p_2)\}$. Suppose on the contrary that $r\in f_G(\alpha)\cap f_G(\beta)\setminus \{f_G(p_1),f_G(p_2)\}$; then
$f_G^{-1}(r)$ is a connected subset of $V(\mathbb D)$ that intersects both $\alpha\setminus \{p_1,p_2\}$ and  $\beta\setminus \{p_1,p_2\}$ contrary to hereditary unicoherence of $\mathbb D$. The claim contradicts the condition that both  $f_G(p_1)$ and $f_G(p_2)$ are points of order at most 3 and not both of them are of order 3.
\end{proof}

\section{Confluent epimorphisms}\label{confluent}

Motivated by the definition of a confluent map between continua (see \cite{JJC-Confluent}) we give an analogous definition for confluent epimorphisms between graphs and develop various tools which will be used in studying these families.

\begin{definition}
Given two topological graphs $G$ and $H$ an epimorphism $f\colon G\to H$  is called {\it confluent} if for every connected subset $Q$ of $V(H)$ and every component $C$ of $f^{-1}(Q)$ we have $f(C)=Q$. Equivalently, if for every connected subset $Q$ of $V(H)$ and every vertex $a\in V(G)$ such that $f(a)\in Q$ there is a connected set $C$ of $V(G)$ such that $a\in C$ and $f(C)=Q$. Clearly, every monotone epimorphism is confluent.
\end{definition}

\begin{observation}
If $f\colon X \to Y$ and $g\colon Y \to Z$ are confluent epimorphisms between topological graphs, then $g\circ f\colon X \to Z$ is a confluent epimorphism.
\end{observation}

\begin{proposition}\label{confluent-edges}
 Given two finite graphs $G$ and $H$ the following conditions are equivalent for an epimorphism $f\colon G\to H$:
\begin{enumerate}
  \item $f$ is  confluent;
  \item for every edge $P\in  E(H)$ and for every vertex $a\in V(G)$ such that $f(a)\in P$, there is an edge $E\in E(G)$ and a connected set $R\subseteq V(G)$ such that  $E\cap R\ne \emptyset$, $f(E)=P$, $a\in R$, and $f(R)=\{f(a)\}$.
  \item  for every edge $P\in  E(H)$ and every component $C$ of $f^{-1}(P)$ there is an edge $E$ in $C$ such that $f(E)=P$.
\end{enumerate}
\end{proposition}

\begin{proof}  (1) $\Rightarrow$ (2): Let $P\in E(H)$ and $a \in V(G)$ such that $f(a) \in P$. Then there exists a connected subset $C\in V(G)$ such that $a \in C$ and $f(C)= P$. Since $C$ is connected there is an edge $E\in C$ such that $f(E)=P$. If $a\in E$ let $R=\{a\}$ otherwise let $R$ be the component of $f^{-1}(f(a))$ containing $a$.  In either case (2) is satisfied.

(2) $\Rightarrow$ (3): This is immediate.

(3) $\Rightarrow$ (1): Given a connected set $Q$ in $H$ start with an edge in $Q$ and use condition (3) edge-by-edge to build a connected set containing one of the end vertices of the initial edge that maps onto $Q$.

\end{proof}

\begin{proposition}\label{confluent-projections}
If $\mathcal F$ is a projective Fra\"{\i}ss\'e family of graphs with confluent epimorphisms and $\mathbb F$ is a projective Fra\"{\i}ss\'e limit of $\mathcal F$, then
for every graph $G\in\mathcal F$ every epimorphism $f_G\colon \mathbb F\to G$ is confluent.
\end{proposition}

\begin{proof}

Let $\{F_n,\alpha_n\}$ be a fundamental sequence for $\mathcal F$. It is easy to see that the proof of \cite[Theorem 2.1.23]{Macias} can be adapted to the present language to show that the canonical projections $\alpha^\infty_m \colon \iLim\{F_n,\alpha_n\} \to F_m$ are confluent. Because $\{F_n,\alpha_n\}$ is a fundamental sequence there is an $m$ and a confluent epimorphism $f\colon F_m \to G$.  Then $f_G= f\circ \alpha_m$ is confluent. 

\end{proof}

\begin{definition}
An epimorphism $f\colon G\to H$ between graphs is called {\it light} if for any vertex $h\in V(H)$
and for any $a,b\in f^{-1}(h)$ we have $\langle a,b \rangle\notin E(G)$.
\end{definition}

\begin{definition}
Recall the diagram (D1).

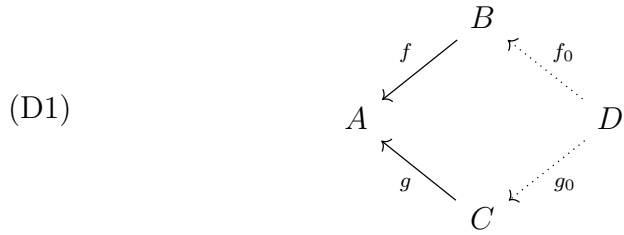
\begin{equation}\tag{D1}
\begin{tikzcd}
&B\arrow{ld}[swap]{f}\\
A&&D\arrow[lu,swap,dotted,"f_0"] \arrow[ld,dotted,"g_0"]\\
&C\arrow[lu,"g"]
\end{tikzcd}
\end{equation}
For given graphs $A,B,C$ and epimorphisms $f\colon B\to A$, $g\colon C\to A$ by {\it standard amalgamation procedure} we mean the graph $D$ defined by $V(D)=\{( b,c)\in B\times C:f(b)=g(c)\}$; $E(D)=\{\langle ( b,c), (b',c') \rangle:\langle b,b'\rangle\in E(B)  \text{ and } \langle c,c'\rangle\in E(C)\}$; and
$f_0(( b,c))=b$, $g_0(( b,c))=c$. Note that $f_0$ and $g_0$ are epimorphisms and $f\circ f_0=g\circ g_0$.
\end{definition}

\begin{proposition}\label{light}
Using the standard amalgamation procedure and the notation of diagram {\rm (D1)}, if the epimorphism $f$
is light, then the epimorphism $g_0$ is light.
\end{proposition}

\begin{proof}
Suppose $g_0$ is not light.  Then there is a vertex $h\in V(C)$ such that  $(a_1,h),(a_2,h)\in g_0^{-1}(h)$, and $\langle (a_1,h),(a_2,h) \rangle\in E(D)$.  Then, $\langle a_1,a_2\rangle \in E(B)$ and $a_1, a_2\in f^{-1}(g(h))$, contradicting $f$ being light.
\end{proof}

\begin{proposition}\label{monotone-implies-monotone}
Using the standard amalgamation procedure and the notation of the diagram {\rm (D1)}, if the epimorphism $f$
is monotone, then the epimorphism $g_0$ is monotone.
\end{proposition}

\begin{proof}
Fix a vertex $c\in V(C)$ and consider $g_0^{-1}(c)=\{\langle x,c\rangle :f(x)=g(c)\}$. This set is isomorphic to
$f^{-1}(g(c))$ which is connected by monotonicity of $f$.
\end{proof}

\begin{proposition}\label{standard-confluent}
Using the standard amalgamation procedure and the notation of diagram {\rm (D1)}, if the epimorphism $f$ is confluent,  then the epimorphism $g_0$ is confluent.
\end{proposition}

\begin{proof}
Let $\langle c,c'\rangle \in E(C)$ and $(b,c)\in V(D)$. By Proposition \ref{confluent-edges} we need to find an edge 
$\langle d_1,d_2\rangle \in D$ that is mapped by $g_0$ onto  $\langle c,c'\rangle$ and a connected set $R$ such that 
$(b,c), d_1\in R$, and $g_0(R)=\{c\}$. Since $f$ is confluent, there are vertices $b_1,b_2\in B$ and a connected set $R_B$ such that $\langle b_1,b_2\rangle\in E(B)$, $b,b_1\in R_B$ and $f(\langle b_1,b_2\rangle)=\langle g(c),g(c')\rangle$. It is enough to put $d_1=( b_1,c)$, $d_2=( b_2,c')$, and $R=R_B\times \{c\}$. 

\end{proof}

\begin{corollary}\label{confluent-components}
Using the standard amalgamation procedure and the notation of diagram {\rm (D1)}, if the graphs $A,B,C$ are connected and the epimorphisms $f$ and $g$ are confluent, then for each component $P$ of the graph $D$ we have $f_0(P)=B$ and $g_0(P)=C$.
\end{corollary}

\begin{corollary}\label{confluent-amalgamation}
Using the standard amalgamation procedure and the notation of diagram {\rm (D1)}, if the graphs $A,B,C$ are connected and the epimorphisms $f$ and $g$ are confluent, then there is a connected graph $D$ and confluent epimorphisms $f_0$ and $g_0$ that make the diagram commutative.
\end{corollary}
\begin{proof}
This follows from Corollary \ref{confluent-components} by taking one of the components as the graph $D$.
\end{proof}

\begin{proposition}\label{m-l-factorization}
Given an epimorphism $f\colon G\to H$ between graphs there is a graph $M$ and epimorphisms
$m\colon G\to M$, $l\colon M\to H$ such that $f=l\circ m$,  $m$ is monotone, and $l$ is light. 
\end{proposition}
\begin{proof}
Define an equivalence relation $\sim$ on $V(G)$ by $x\sim y$ if and only if $f(x)=f(y)$ and $x$ and $y$ are in the same
component of $f^{-1}(f(x))$. Putting $V(M)=V(G)/\sim$, $m\colon V(G)\to V(M)$ as the projection,
$\langle [x]_\sim ,[y]_{\sim} \rangle\in E(M)$
if and only if $\langle x,y\rangle\in E(G)$, and $l([x]_\sim)=f(x)$ one can verify that the conclusions of the Proposition are satisfied.
\end{proof}

\begin{proposition}\label{composition-confluent}
Consider the epimorphisms $f\colon F\to G$ and $g\colon G\to H$ between topological graphs. If the composition $g\circ f\colon F\to H$ is confluent, then $g$ is confluent.
\end{proposition}
\begin{proof}
Let $E$ be a closed connected subset of $H$, and let $a\in G$ be a vertex such that $g(a)\in E$. We need to find a closed connected subset of $S$ of $G$ that contains $a$ and such that $g(S)=E$. Since $g\circ f$ is confluent, there is a closed connected subset $C$ of $F$ that contains a vertex in $f^{-1}(a)$ and such that $(g\circ f)(C)=E$. Then $S=f(C)$ satisfies the requirements.
\end{proof}  

\begin{corollary}\label{m-l-factorization-confluent}
Under the hypothesis of Proposition \ref{m-l-factorization} if the epimorphism $f$ is confluent, then both $m$ and $l$ are confluent.
\end{corollary}
\begin{proof}
   The epimorphism $m$ is confluent since it is monotone, while $l$ is confluent by Proposition \ref{composition-confluent}.
\end{proof}

\section{Rooted trees}\label{rooted-trees}

The following example shows that confluent epimorphisms on trees does not behave well.

\begin{example} \label{no-confluent}Let $A$ consist of just one edge $\langle 0,1\rangle$, $B$ be an arc with three vertices $\{a,b,c\}$ and two edges $\langle a,c\rangle$, $\langle c,b\rangle$, and $C$ be again an arc with three vertices
$\{p,q,r\}$ and two edges $\langle p,r\rangle$, $\langle r,q\rangle$. Let $f\colon B\to A$ be such that $f(a)=f(b)=0$
and $f(c)=1$ and let $g\colon C\to A$ be such that $g(p)=g(q)=1$ and $g(r)=0$, see figure below.
It is easy to see that $f$ and $g$ are confluent epimorphisms.

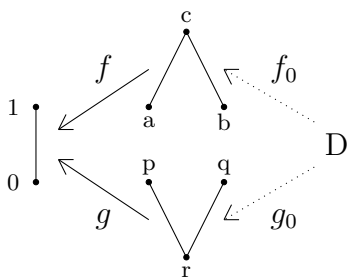
\begin{center}
\begin{tikzpicture}[scale=1]

  \draw (0,1) -- (0,2);
  \filldraw[black] (0,1) circle (1pt);
   \filldraw[black] (0,2) circle (1pt);
  \draw (1.5,2) -- (2,3) -- (2.5,2);
     \filldraw[black] (1.5,2) circle (1pt);
     \filldraw[black] (2,3) circle (1pt);
    \filldraw[black] (2.5,2) circle (1pt);
  \draw (1.5,1) -- (2,0) -- (2.5,1);
     \filldraw[black] (1.5,1) circle (1pt);
    \filldraw[black] (2,0) circle (1pt);
    \filldraw[black] (2.5,1) circle (1pt);
 \node at (-0.3,2) {\scriptsize 1};
 \node at (-0.3,1) {\scriptsize 0};
 \draw (1.5,2.5) --  (0.3,1.7);
 \draw  (0.3,1.7) --(0.4,1.9);
  \draw  (0.3,1.7) --(0.5,1.7);
  \node at (1.5,1.8) {\scriptsize a};
  \node at (2.5,1.8) {\scriptsize b};
\node at (2,3.2) {\scriptsize c};
 \draw (1.5,0.5) --  (0.3,1.3);
 \draw  (0.3,1.3) --(0.5,1.3);
  \draw  (0.3,1.3) --(0.4,1.1);
 \draw [dotted] (3.7,1.8) --  (2.5,2.5);
 \draw  (2.5,2.5) --(2.7,2.5);
  \draw  (2.5,2.5) --(2.6,2.3);
  \node at (4,1.5) { D};

 \draw [dotted] (3.7,1.2) --  (2.5,0.5);
 \draw  (2.5,0.5) --(2.7,0.5);
  \draw  (2.5,0.5) --(2.6,0.7);

 \node at (0.9,2.5) {$f$};
 \node at (0.9,0.5) {$g$};
  \node at (3.3,2.5) {$f_0$};
 \node at (3.3,0.5) {$g_0$};

  \node at (1.5,1.2) {\scriptsize p};
  \node at (2.5,1.2) {\scriptsize q};
\node at (2,-0.2) {\scriptsize r};

\end{tikzpicture}
\end{center}

Note that the standard amalgamation procedure gives the graph $D$ to be the cycle $( c,p)$, $( a,r)$, $( c,q)$, $( b,r)$, $( c,p)$ hence not a tree. 

In fact, no amalgamation of the trees $A$, $B$, and $C$ yields a tree.  To see this assume there is an amalgamation of $A$, $B$, and $C$, that is, there is finite tree $D$ and confluent epimorphisms $f_0\colon D\to B$ and $g_0\colon D \to C$ such that $g\circ g_0 = f\circ f_0$.

We will show there must then exist an infinite sequence of distinct vertices in the finite tree $D$.  To this end, let $\langle a_1,p_1\rangle$ be an edge in $D$ such that $f_0$ maps onto the edge $\langle a,c\rangle$ in $B$ with $f_0(a_1)=a$,  $f_0(p_1)=c$ and $g_0(p_1)=p$.  We can select a vertex $p_1$ that satisfies the last equality since $f(c)=1$. Next let $\langle p_1,b_1\rangle$ be an edge in $D$ that $g_0$ maps onto the edge $\langle p,r\rangle$ in $C$. So $g_0(b_1)=r$ and we may select $b_1$ so that $f_0(b_1)=b$. Note that $a_1\not = b_1$. In the same manner we can select a vertex $q_1$ in $D$ such that $f_0$ maps the edge $\langle b_1,q_1\rangle$ onto the edge $\langle a,c\rangle$ in $B$ with $g_0(q_1)=q$ and $q_1\not = p_1$. Inductively we obtain sequences of vertices $(a_n)$, $(b_n)$, $(p_n)$, and $(q_n)$ where unicoherence of the tree $D$ guarantees that no vertices in  a sequence equals a vertices earlier in the sequence. Since $D$ is a finite tree we have a contradiction to our assumption that there was an amalgamation of $A$, $B$, and $C$.

\end{example}
Because of Example \ref{no-confluent} we will consider rooted trees and whenever we have an epimorphism between rooted trees we will assume that the epimorphism
preserves orders defined by the selected points. Precise definitions are given below.

\begin{definition} 

By a {\it rooted graph} we mean a finite graph $G$ with a distinguished vertex $r(G)$. By a {\it rooted tree} we mean a rooted graph which is a tree.
On a rooted tree
we define an order $\le$ by $x\le y$ if every arc containing $r(T)$ and $y$ contains $x$. 
We require epimorphisms $f\colon T\to S$ between rooted trees to preserve orders, in particular
$f(r(T))=r(S)$.
\end{definition}

Epimorphisms between rooted trees need not be confluent as the following example shows.

\begin{example}
Let $S$ be a triod with vertices $\{A,B,C,D\}$, edges $\{\langle A,B\rangle,\langle B,C\rangle, \langle B,D\rangle\} $, and $A$ is the root of $T$.  
Let $T$ be another triod with vertices $\{a, b_1,b_2,c, d_1,d_2\}$, edges $\{\langle a,b_1\rangle, \langle a,b_2\rangle, \langle b_1,c\rangle, 
\langle b_1,d_1\rangle, \langle b_2,d_2\rangle\}$, and $a$ is the root of $S$.  Let $f\colon T \to S$ be given by $f(a)=A$, $f(b_1)=f(b_2)= B$, $f(c)=C$, 
and $f(d_1)=f(d_2)=D$.  Then $f$ is an order preserving epimorphism on rooted trees but is not confluent.

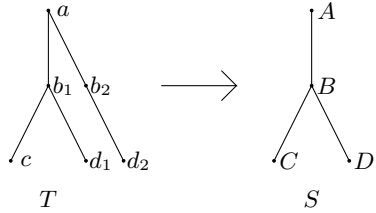
\begin{center}
\begin{tikzpicture}[scale=0.5]

  \draw (0,0) -- (1,2);
  \filldraw[black] (0,0) circle (1pt);
  \filldraw[black] (1,2) circle (1pt);
  \draw (2,0) -- (1,2);
  \node at (0.4,0) {\scriptsize $c$};
  \node at (2.4,0) {\scriptsize $d_1$};  
  \filldraw[black] (1,4) circle (1pt);
\filldraw[black] (2,0) circle (1pt);
  \node at (1.4,4) {\scriptsize $a$};
    \node at (1.4,2) {\scriptsize $b_1$};
\node at (2.4,2) {\scriptsize $b_2$};
\draw (1,2) -- (1,4); 
\draw (1,4) -- (2,2);  
\draw (3,0) -- (2,2); 
\filldraw[black] (2,2) circle (1pt);  
\filldraw[black] (3,0) circle (1pt); 
\node at (3.4,0) {\scriptsize $d_2$};  
\draw (4,2) -- (6,2) -- (5.6, 2.4); 
 \draw (6,2) -- (5.6, 1.6);  
 
\draw (7,0) -- (8,2) -- (8,4);
\draw (8,2) -- (9,0);  
\filldraw[black] (7,0) circle (1pt); 
\filldraw[black] (8,2) circle (1pt); 
\filldraw[black] (8,4) circle (1pt); 
\filldraw[black] (9,0) circle (1pt); 
\node at (8.4,4) {\scriptsize $A$};
 \node at (8.4,2) {\scriptsize $B$};
 \node at (7.4,0) {\scriptsize $C$};
 \node at (9.4,0) {\scriptsize $D$}; 
 \node at (1,-1) {\scriptsize $T$};
 \node at (8,-1) {\scriptsize $S$}; 
\end{tikzpicture}
\end{center}
\end{example}

\begin{definition}
By an {\it ordered graph} $G$ we mean finite rooted graph with a partial order $\le$ such that:
\begin{itemize}
    \item the root $r(G)$ is the least element in $V(G)$ i.e. $r(G)\le x$ for each $x\in V(G)$;
    \item if $\langle a,b\rangle\in E(G)$, $a\not = b$ then $a \le b$ or $b\le a$ and if 
$c\in V(G)\setminus \{a,b\}$ then $a \le c \le b$ is not true.
\end{itemize}
\end{definition}

The standard amalgamation of rooted trees with order preserving epimorphisms need not be a tree as the following example shows.

\begin{example}\label{rooted-standard-not-tree}
Let $A$ consist of a single vertex, $B$ and $C$ be edges then the standard amalgamation gives a complete graph on four vertices.
\end{example}

\begin{definition} \label{chain-def}

By a {\it chain} in an  ordered graph $G$ we mean sequence $a_1, a_2, \dots a_n$ of vertices of $G$ such that:

\begin{enumerate}
\item $a_1=r(G)$;
\item $a_1\le a_2\le \dots \le a_n$.
\end{enumerate}
A chain is {\it proper} if moreover it satisfies the condition
\begin{enumerate}
\setcounter{enumi}{2}
    \item $a_1\ne a_2\ne \dots \ne a_n$.
\end{enumerate}
\end{definition}

\begin{lemma}\label{amalgamation-light}
Consider epimorphisms $f\colon B\to A$ and $g\colon C\to A$ between rooted trees (in particular preserving the orders) and suppose $f$ is light. Then the standard amalgamation procedure gives a rooted tree $D$, epimorphisms
$f_0\colon D\to B$, $g_0\colon D\to C$, and the epimorphism $g_0$ is light.
\end{lemma}

\begin{proof}
The lightness of the epimorphism $g_0$ is a consequence of Proposition \ref{light}. First we show that
$D$ is a rooted tree. We put the root of $D$ at $((r(B),r(C))$. Suppose on the contrary, that the graph $D$ is not a tree. Then there are two proper chains $(b_1,c_1), (b_2,c_2)\dots (b_n,c_n)$ and $(b'_1,c'_1), (b'_2,c'_2)\dots (b'_m,c'_m)$ of vertices of $D$ such that:

\begin{enumerate}
    \item $(b_n, c_n)=(b'_m, c'_m)$;
    \item there is an index $k$ such that $(b_k, c_k)\ne (b'_k, c'_k)$.
\end{enumerate}
We are heading toward contradiction by two Claims.

Claim 1. The chains $(b_1,c_1), (b_2,c_2),\dots (b_n,c_n)$ and $(b'_1,c'_1), (b'_2,c'_2),\dots$ $(b'_m,c'_m)$ have the same length, i.e $n=m$, and $c_1=c'_1$, $c_2=c'_2, \dots$ $c_n=c'_n$. To see this note that the epimorphism $g_0$ is just the projection onto the second coordinate and it is light, therefore it is one-to-one on chains. Since there is only one chain in $C$ joining $r(C)$ and $c_n$ we have $c_1=c'_1$, $c_2=c'_2, \dots c_n=c'_n$.

Claim 2. If $(b_k,c_k)\ne (b'_k,c_k)$, then $b_k$ and $b'_k$ are incomparable, i.e. $b_k\not\le b'_k$ and $b'_k\not\le b_k$. Assume on the contrary that  $b_k\le b'_k$. Then, by commutativity of the diagram (D1) we have $f(b_k)=f(b'_k)=g(c_k)$ and because $f$ preserves orders
we have also $f(b_k)=f(b_{k+1})=\dots=f(b'_k)=g(c_k)$, which contradicts the lightness of $f$.

We then have a contradiction to the fact that $B$ is a tree, since $b_1, b_2,\dots b_n$ and $b'_1, b'_2,\dots b'_n$ are two chains in $B$
from $r(B)$ to $b_n$ passing through two incomparable points $b_k$ and $b'_k$ respectively.

Finally, if follows from the definition of the standard amalgamation procedure that $f_0$ and $g_0$ preserve order.
\end{proof}

\begin{proposition}\label{m-l-factorization-pointed}
Under the hypothesis of Proposition \ref{m-l-factorization} if $G$ and $H$ are rooted trees
and $f$ is an epimorphism between rooted trees, i.e. it preserves the order, then $M$ is a rooted tree, and the epimorphisms $m$ and $l$ preserve orders.
\end{proposition}
\begin{proof}
First, observe that hereditary unicoherence of trees implies that every monotone epimorphism preserves orders. This implies that $M$ is a rooted tree and that $m$ is an epimorphism between rooted trees. Since $f$ and $m$ preserve orders, $l$ preserves the order as well.
\end{proof}

The family of rooted trees with order preserving epimorphisms were considered in  \cite{B-K}. The authors show there that the topological
realization of the
projective Fra\"{\i}ss\'e limit of the family is homeomorphic to the Lelek fan.

\begin{observation}\label{monotone-relative}
If $f^G_H\colon G\to H$ is an epimorphism between rooted trees and $C$ is a chain in $G$, then $f^G_H|_C$ is monotone.
\end{observation}

\begin{proposition}\label{two-orders}
Let $\mathcal F$  be a projective Fra\"{\i}ss\'e family of rooted trees and let $\mathbb T$ be a projective Fra\"{\i}ss\'e limit of $\mathcal F$.
For any $x,y\in \mathbb T$, $x\le y$ if and only if for any $T\in\mathcal T$ and for any epimorphism $f_T\colon \mathbb T\to T$,
$f_T(x)\le f_T(y)$.
\end{proposition}
\begin{proof}
Denote by $\mathbb T$ the projective Fra\"{\i}ss\'e limit of $\mathcal T$. By the root $r(\mathbb T)$ we denote the only point of $\mathbb T$ such that for any tree $T\in \mathcal T$ and any epimorphism $f_T\colon \mathbb T\to T$ we have $f_T(r(\mathbb T))=r(T)$.

First assume for any epimorphism $f_T\colon \mathbb T \to T$ that $f_T(x)\le f_T(y)$. Let $I_T$ be the arc from $f_T(r(T))$ to $f_T(y)$. Then $f_T(x) \in I_T$ because $f_T(x) \le f_T(y)$.  The family $\mathcal A=\{I_T: T\in \mathcal F\}$ forms a projective Fra\"{\i}ss\'e family with monotone epimorphisms by Observation \ref{monotone-relative}. By Proposition \ref{monotone-limit-of-arcs} the projective Fra\"{\i}ss\'e limit of $\mathcal A$ is an arc from $r(\mathbb T)$ to $y$ that contains $x$, so $x\le y$ as needed.

To see the other implication let $\{T_n,\alpha_n\}$ be a fundamental sequence for $\mathbb T$. Take $x,y \in \mathbb T$ such that $x \le y$ and suppose, by contradiction, that there is an $m$ such that $x_m \not \le y_m$. For any $n$ let $I_n$ be a chain in $T_n$ from $r(T_n)$ to $y_n$. Then $x_m \not \in I_m$. Since the epimorphisms in $\mathcal T$ preserve order, we have $\mathbb I=\iLim\{I_n,\alpha_n|_{I_n}\}$ is an arc in $\mathbb T$ containing $y$ and $r(\mathbb T)$. Since $x \le y$, $x \in \mathbb I$, so $x_m \in I_m$ contrary to the assumption that $x_m \not \in I_m$. For any $T \in \mathcal T$
 and $f_T\colon {\mathbb T} \to T$  there is an $m$ and a monotone epimorphism $f\colon T_m \to T$ such that $f_T=f\circ \alpha_m$ so $f_T(x)=f\circ \alpha_m(x) \le f\circ \alpha_m(y)=f_T(y)$.

\end{proof}

\begin{corollary}\label{arwise-connected}
If $\mathcal T$ is a projective Fra\"{\i}ss\'e family of rooted trees, and $\mathbb T$ is a projective Fra\"{\i}ss\'e 
limit of $\mathcal T$, then $\mathbb T$ is arcwise connected.
\end{corollary}
\begin{proof}
By Proposition~\ref{two-orders} the root $r(\mathbb T)$ is the least element in the order $\le$ on $\mathbb T$, in particular there is an arc in $\mathbb T$ joining $r(\mathbb T)$ to any point in $\mathbb T$.
\end{proof}

By Theorem \ref{limit-of-hu} and Corollary \ref{arwise-connected} we get the following Corollary.

\begin{corollary}\label{dendroid}
If $\mathcal T$ is a projective Fra\"{\i}ss\'e family of rooted trees, and $\mathbb T$ is a projective Fra\"{\i}ss\'e
limit of $\mathcal T$, then $\mathbb T$ is a dendroid.
\end{corollary}

Recall the following definition for continua.
\begin{definition}
A continuum $X$ which is a dendroid is said to be a {\it smooth dendroid} if there exists a point $p\in X$ such that if $x_n\in X$ is a sequence of points that converges to a point $x\in X$ then the sequence of arcs $px_n$ converges to the arc $px$.  This is equivalent to saying the order $\le$, defined by $x\le y$ if every arc joining $p$ and $y$ contains $x$, is closed.
\end{definition}

We now give an analogous definition for topological rooted graph to be a smooth dendroid.

\begin{definition}
A topological rooted graph $X$ with the root $r(X)$ is called a {\it smooth dendroid} if $X$ is a dendroid according to Definition \ref{dendroid-def} and the order $\le$, defined by $x\le y$ if every arc joining $r(X)$ and $y$ contains $x$, is closed.
\end{definition}

\begin{observation}
If a rooted topological graph $G$ is a smooth dendroid and $E(G)$ is transitive, then its topological realization $|G|$ is a smooth dendroid in the topological sense.
\end{observation}

\begin{corollary}\label{smooth-dendroid}
If $\mathcal T$ is a projective Fra\"{\i}ss\'e family of rooted trees, and $\mathbb T$ is a projective Fra\"{\i}ss\'e 
limit of $\mathcal T$, then $\mathbb T$ is a smooth dendroid.
\end{corollary}

\begin{proof}
First note that by Corollary \ref{dendroid} $\mathbb T$ is a dendroid. It follows from  Proposition~\ref{two-orders} that the order $\le$
is closed on $\mathbb T$, so the dendroid is smooth.
\end{proof}

The Lelek fan can be characterized as a smooth fan with a dense set of end points, \cite{Lelek-fan}. In  \cite{B-K}, it is shown that the topological realization of the Fra\"{\i}ss\'e limit for the family of all rooted trees with order preserving epimorphisms is the Lelek fan.

\section{Kelley topological graphs}

We first recall definitions concerning Kelley continua and adapt these to the setting of topological graphs in order to prove properties of the topological realization $|\mathbb{T}_{\mathcal C}|$ where $\mathbb{T}_{\mathcal C}$ is the projective Fra\"{\i}ss\'e limit of $\treecon$. 

\begin{notation}
For a given metric space $X$ with a metric $d$, the $r$-neighborhood of a closed set $A\subseteq X$ is
$N(A,r)=\{x\in X: \text{ there is } a\in A:d(x,a)<r\} $.
The Hausdorff distance $H$ between closed subsets of $X$ is
defined by $H(A,B)=\inf\{r>0: A\subseteq N(A,r) \text{ and } B\subseteq N(A,r)\}$.
\end{notation}

\begin{definition}
A continuum $X$ is said to be {\it Kelley continuum} if for every subcontinuum $K$ of $X$,
every $p\in K$, and every sequence $p_n\to p$ in $X$ there are subcontinua $K_n$ of $X$ such that $p_n\in K_n$ and
$\lim K_n=K$. By compactness, this definition is equivalent to saying that for every $\varepsilon>0$ there is $\delta>0$ such that for each two points $p,q\in X$ satisfying $d(p,q)<\delta$, and for each subcontinuum $K\subseteq X$ such that $p\in K$ there is a subcontinuum $L\subseteq X$ satisfying $q\in L$ and $H(K,L)<\varepsilon$.
\end{definition}

\begin{definition}
A topological  graph $X$ is called {\it Kelley} if $X$ is connected and for every closed and connected set $K\subseteq V(X)$,
every vertex $p\in K$, and every sequence $p_n\to p$ of vertices in $X$ there are closed and connected sets $K_n$ such that $p_n\in K_n$ and
$\lim K_n=K$. By compactness, this definition is equivalent to saying that for every $\varepsilon>0$ there is $\delta>0$ such that for each two vertices $p,q\in X$ satisfying $d(p,q)<\delta$, and for each connected set $K\subseteq X$ such that $p\in K$ there is a connected set $L\subseteq X$ satisfying $q\in L$ and $H(K,L)<\varepsilon$. Note that every finite graph is Kelley.

\end{definition}

\begin{observation}\label{Kelley-observation}
If a topological graph $G$ is Kelley and $E(G)$ is transitive, then its topological realization $|G|$ is a Kelley continuum.
\end{observation}
\begin{proposition}\label{Kelley-theorem}
If $\mathcal G$ is a projective Fra\"{\i}ss\'e family of graphs with confluent epimorphisms, then the projective Fra\"{\i}ss\'e limit is Kelley.
\end{proposition}

\begin{proof}

Let $\{G_n,\alpha_n\}$ be a fundamental sequence for $\mathcal G$ so ${\mathbb G} = \iLim\{G_n,\alpha_n\}$. The proof that the inverse limit of Kelley continua with confluent bonding maps in \cite[Theorem 3.1]{Ingram} is a Kelley continuum can be adapted to the present setting to show that $\mathbb G$ is Kelley.

\end{proof}

\section{End vertices and ramification vertices in topological graphs}

In this section we define ramification and end vertices for topological graphs and study how they project onto the topological realiztion of a projective Fra\"{\i}ss\'e limit.

\begin{definition}\label{rampt-topo-dend}
For a topological graph $D$ that is a dendroid, we say that a vertex $v$ has order at least $n$, in symbols $\ord(v)\ge n$,
if there are arcs $A_1, A_2, \dots A_n$ such that $A_i\cap A_j=\{v\}$ for $i,j\in \{1,2, \dots n\}$ satisfying $i\ne j$. We
define $\ord(v)=n$ if $\ord(v)\ge n$ and $\ord(v)\ge n+1$ is not true. If $\ord(v) \ge n$ for each positive integer $n$, then we say that $v$ has infinite order, in symbols $\ord(v)=\infty$.
Vertices of order 1 are called end vertices and vertices of order $\ge 3$ are called ramification vertices.

Note that these definitions agree with Definitions \ref{definition-order} in case of finite graphs. We call an end vertex {\it isolated} if it belongs to an edge. Note that, according to this definition, every end vertex in a finite graph is isolated.
\end{definition}

Let us recall a continuum theory definition.

\begin{definition}
For a continuum $X$ a point $p\in X$ is called a {\it ramification point of $X$ in the classical sense} if there are three arcs $A,B,C$ in $X$ such that
$A\cap B=A\cap C=B\cap C=\{p\}$.
Similarly, a point $p\in X$ is called an {\it end point of $X$ in the classical sense} if it is an end point of every arc in $X$ that contains $p$.
\end{definition}

\begin{proposition}\label{endpts-to-endpts}
Suppose  $D$ is a topological graph that is a dendroid, the set of edges of $D$ is transitive and $e$ is a non-isolated end vertex of $D$. Let $\varphi$ be the quotient map from $D$ onto its topological realization $|D|$. Then $\varphi(e)$ is an end point of $|D|$ in the classical sense.
\end{proposition}

\begin{proof}
 Assume the conclusion is not true.  Then there exists an arc $A \subseteq |D|$ such that $A\setminus\{\varphi(e)\}=H \cup K$ where $H$ and $K$ are non-empty separated sets.
The set $\varphi^{-1}(H\cup K)$ is not connected since the image of a connected set is connected. However the sets $\varphi^{-1}(H)$ and $\varphi^{-1}(K)$ are each connected.  Let $a$ and $b$ be vertices in $\varphi^{-1}(H)$ and $\varphi^{-1}(K)$ respectively and $A_H$ and $A_K$ be arcs in $D$ from $a$ to $e$ and $b$ to $e$. The arcs $A_H$ and $A_K$ then lie in $\varphi^{-1}(H)\cup \{e\}$ and $\varphi^{-1}(K)\cup \{e\}$ respectively.  The intersection of $A_H$ and $A_K$ is $\varphi^{-1}(\varphi(e))=\{e\}$ since $e$ is a non-isolated end vertex of $D$. So $\varphi^{-1}(H)\cup \varphi^{-1}(K)\cup\{e\}$ is an arc containing $e$ but $e$ is not the end vertex of this arc. Thus $e$ is not an end vertex of $D$ contrary to the hypothesis.
\end{proof}

\begin{proposition}\label{end-to-end}
Let $f\colon X\to Y$ be an order preserving epimorphism between smooth dendroids. Then, for every end vertex $y\in Y$ there is an end vertex $x\in X$ such that $f(x)=y$.
\end{proposition}
\begin{proof}
Let $z$ be any vertex in $X$ such that $f(z)=y$, and let $x$ be an end vertex of $X$ satisfying $z\le x$. Since $f$ is order preserving we have $f(z)=y\le f(x)$, but $y$ is an end, point, so $f(z)=f(x)=y$.
\end{proof}

\begin{proposition}\label{ram-to-ram}
Suppose  $D$ is a topological graph that is a dendroid, the set of edges of $D$ is transitive and $r$ is a ramification vertex  of $D$. Let $\varphi$ be the quotient map from $D$ onto its topological realization $|D|$. If $\varphi^{-1}(\varphi(r))$ contains no isolated end vertices,  then $\varphi(r)$ is a ramification point of $|D|$ in the classical sense.
\end{proposition}

\begin{proof}
Suppose $r$ is a ramification vertex of $D$ and let $A$, $B$, and $C$ be arcs in $D$ such that $A\cap B=B\cap C=C\cap A=\{r\}$. The images $\varphi(A)$, $\varphi(B)$, and $\varphi(C)$ are arcs or degenerate with pairwise intersections equaling $\varphi(r)$.  Suppose $\varphi(A)$ is degenerate then $A$ consists of a single edge $\langle r,e\rangle$ and $A=\varphi^{-1}(\varphi(r))$ contains the isolated end vertex $e$ contrary to the hypothesis.  Thus $\varphi(r)$ is a ramification point of $|D|$.

\end{proof}
The following Observation follows from Proposition~\ref{two-orders}.

 \begin{observation}\label{end-to-end-rooted}
Suppose that $\mathcal T$ is a projective Fra\"{\i}ss\'e family of rooted trees and that a vertex $e\in \mathbb{T}$ has the property that for every $T\in \mathcal T$ and every epimorphism
$f_T\colon \mathbb{T}\to T$ the image $f_T(e)$ is an end vertex of $T$. Then $e$ is an end vertex of
$\mathbb T$.
\end{observation}

\begin{theorem}\label{rooted-end}
Suppose that $\mathcal T$ is a projective Fra\"{\i}ss\'e family of rooted trees and for every $T\in \mathcal T$, for every end vertex $e\in V(T)\setminus \{r(T)\}$  and every $a\in V(T)$ such that $\langle a,e\rangle\in E(T)$ there is a tree $S$, an epimorphism $f^S_T\colon  S\to T$  such that for every $p\in (f^S_T)^{-1}(a)$ there are vertices $q,r\in V(S)$ satisfying $p\le q\le r$, $q\ne r$, and $f^S_T(q)=f^S_T(r)=e$.
Then the projective Fra\"{\i}ss\'e limit of $\mathcal T$ has no isolated end vertices.
\end{theorem}

\begin{proof}
Let $\mathcal T$ be a family that satisfies the assumptions of the Theorem, and let $\mathbb T$ be the projective Fra\"{\i}ss\'e limit of $\mathcal T$. Suppose $e$ is an isolated end vertex of $\mathbb T$ then there exists $a\in V(\mathbb T)$ such that $\langle a,e\rangle\in E(\mathbb T)$. Let $T\in \mathcal T$ and $f_T:\mathbb T \to T$ be an epimorphism such that $f_T(a)\not = f_T(e)$ then $\langle f_T(a),f_T(e)\rangle \in E(T)$.   From the hypothesis we know there is a tree $S$ and an epimorphism  $f^S_T\colon  S\to T$ such that for every $p\in (f^S_T)^{-1}(f_T(a))$ there are vertices $q,r\in V(S)$ satisfying $p\le q\le r$, $q\ne r$, and $f^S_T(q)=f^S_T(r)=f_T(e)$. Let $f_S\colon \mathbb T\to S$ be an epimorphism such that $f_T=f^S_T\circ f_S$. Then $f_S(a)\in(f^S_T)^{-1}(f_T(a))$ and $r=f_S(e)$, since $f_S$ is an order preserving epimorphism, but these vertices are not members of an edge in $S$ contrary to the fact that $f_S$ is an epimorphism.
\end{proof}

\section{Rooted trees and epimorphism preserving end vertices}

In this section we consider order preserving epimorphisms on trees which also preserve end vertices.

\begin{definition}
We say that an epimorphism $f\colon S\to T$ between rooted trees {\it preserves end vertices} if, beside preserving order, we have that an image of an end vertex in $S$ is an end vertex in $T$.
\end{definition}

\begin{definition} 
 A {\it branch} of a rooted tree $T$ is a maximal chain in $T$, the {\it length} of a branch is the number of vertices in the branch. 
 \end{definition}
 
\begin{definition}
A rooted tree $T$ is a {\it fan} if for every $s,t\in T$ which are incomparable in the order, if $p\not \in \{s,t\}$ such that $\langle p,s\rangle$ and $\langle p,t\rangle$ are in $E(T)$ the $p$ is the root of $T$. 
\end{definition}

\begin{definition}
A fan $F$ is {\it uniform} if the length of all branches in $F$ are equal.
\end{definition}

The following observation was noted in \cite[Remark 2.2]{B-K}.

\begin{observation}\label{trees-to-fans}
For any rooted tree $T$ there is a uniform fan $S$ and an epimorphism preserving end vertices $f\colon S\to T$. 
\end{observation}
 
\begin{observation}\label{increase-lengths}
For any uniform fan $F$ and a number $n$ bigger then the length of branches in $F$ there is a uniform fan $F'$ and an epimorphism preserving end vertices $f\colon F'\to F$ such that the lengths of all branches in $F'$ are bigger than $n$.
\end{observation}

\begin{definition}
Denote by $\mathcal T_{\mathcal E}$ the family of rooted trees and epimorphisms preserving end vertices. Similarly, 
$\mathcal F_{\mathcal E}$ is the family of uniform fans and  epimorphisms preserving end vertices.
\end{definition}

\begin{theorem}\label{isomorphic}
The families $\mathcal T_{\mathcal E}$ and $\mathcal F_{\mathcal E}$ are projective Fra\"{\i}ss\'e families and their Fra\"{\i}ss\'e limits are isomorphic.
\end{theorem}

\begin{proof}
For amalgamation in  $\mathcal T_{\mathcal E}$ first move to $\mathcal F_{\mathcal E}$ by Observation \ref{trees-to-fans}, then pass to fans with the same lengths of branches by Observation \ref{increase-lengths}, and then amalgamate the fans. By Observation \ref{trees-to-fans} the family $\mathcal F_{\mathcal E}$ is cofinal in $\mathcal T_{\mathcal E}$, so their limits 
are isomorphic. 
\end{proof}

\begin{theorem}\label{smooth-dendroid-fan}
The projective Fra\"{\i}ss\'e limit $\mathbb F_{\mathcal E}$ of $\mathcal F_{\mathcal E}$ is a smooth dendroid with only one ramification vertex, namely $r(\mathbb F_{\mathcal E})$.
\end{theorem}
\begin{proof}
Suppose on the contrary that there is vertex $q\in \mathbb F_{\mathcal E}\setminus r(\mathbb F_{\mathcal E})$ that is a ramification vertex. Then there are arcs $A$, $B$, and $C$ such that $A\cap B=A\cap C=B\cap C=\{q\}$. Choosing $\varepsilon$ small enough and an $\varepsilon$-epimorphism $f_G\colon \mathbb F_{\mathcal E}\to G$ we have that $f_G(A)$, $f_G(B)$, and $f_G(C)$ are three connected subsets of $G\setminus r(G)$
such that $f_G(A)\not\subseteq f_G(B)\cup f_G(C)$, $f_G(B)\not\subseteq f_G(A)\cup f_G(C)$, and 
$f_G(C)\not\subseteq f_G(B)\cup f_G(C)$. This contradicts the fact that $G$ is a fan.
\end{proof}

Let us recall the notion of a dense in itself set.
\begin{definition}
A metric space $(X,d)$ is called {\it dense in itself} if for every point 
$p\in X$ and for every $\varepsilon >0$ there is a point $q\in X$, $q\ne p$ such that $d(p,q)<\varepsilon$.
\end{definition}

\begin{theorem}\label{end-to-end-rooted-reverse}
A vertex $e\in \mathbb{T}_{\mathcal{E}}$ is an end vertex if and only if for every $T\in \mathcal{T}_{\mathcal{E}} $ and every epimorphism
$f_T\colon \mathbb{T}_{\mathcal {E}}\to T$ the image $f_T(e)$ is an end vertex of $T$.
\end{theorem}

\begin{proof}
One direction is just Observation \ref{end-to-end-rooted}.
To see the other implication, suppose on the contrary that $e$ is an end vertex of $\mathbb{T}_{\mathcal {E}}$, $f_W\colon \mathbb{T}_{\mathcal {E}}\to W$ is an epimorphism and  $f_W(e)$ is an not an end vertex of $W$.

Consider an arc  $A$ in $W$ that satisfies the following conditions:
\begin{enumerate}
    \item $f_W(e)\in A$;
    \item $f_W(e)\le x$ for every $x\in A$;
\end{enumerate}

By \cite[Theorem 2.4]{Pseudo} there is a sequence $T_1=W, T_2, \dots$ of members of $\mathcal T_{\mathcal {E}}$  and epimorphisms $f_1,f_2,\dots$ with
$f_i\colon T_{i+1}\to T_i$ such that the inverse limit space $\iLim \{T_i,f_i\}$ is isomorphic to $\mathbb{T}_{\mathcal {E}}$.  We define, by induction, a sequence
$A_1,A_2,\dots$ of arcs of $T_1, T_2,\dots$, with end vertex $f_{T_i}(e)$ and $e_i$ such that the following holds:

\begin{enumerate}
    \item $A_1=A$;
    \item $f_i(A_{i+1})=A_i$;
    \item $f_{T_i}\colon \mathbb T_{\mathcal {E}}\to T_i$ is an epimorphism such that $f_{T_i}=f_i\circ f_{T_{i+1}}$;
    \item $f_{T_{i+1}}(e)\le x\le e_{i+1}$ for every $x\in A_{i+1}$.
    \item $f_i|_{A_{i+1}}$ is a monotone epimorphism;
\end{enumerate}

Assume $A_i$'s and $e_i$'s for $1\le i <n$ are found so that these conditions are satisfied. Let $C$ be the component of $f_n^{-1}(A_n)$ that contains $f_{T_{n+1}}(e)$, and define $D=\{x\in C: f_{T_{n+1}}(e)\le x\}$. Since $f_n(f_{T_{n+1}}(e))$ is not an end vertex of $T_n$, the set $D$
 is nondegenerate. Let $e_{n+1}$ be an end vertex in $D$ and define $A_{n+1}$ to be the arc from $e_{n+1}$ to
$f_{T_{n+1}}(e)$. Note that $e_{n+1}$ is an end vertex of $T_{n+1}$; since the epimorphisms $f_n$ preserve end vertices we have $f_n(e_{n+1})=e_n$.
Then the image $f_n(A_{n+1})=A_n$ and conditions (1)-(5) are satisfied for $i\le n$.

Finally, by Proposition~\ref{two-orders},   $\iLim \{A_n, f_n|_{A_{n+1}}\}$ is a  monotone arc with end vertices $e$ and $\langle e_1, e_2\dots \rangle$, contrary to the fact that $e$ was an end vertex of $\mathbb T_{\mathcal {E}}$.
\end{proof}

\begin{theorem}\label{no-isolated-ends}
The set of end vertices of the projective Fra\"{\i}ss\'e limit $\mathbb T_{\mathcal E}$ of $\mathcal T_{\mathcal E}$ is dense in itself.
\end{theorem}
\begin{proof}
Let $p$ be an end vertex of $\mathbb T_{\mathcal E}$, let $\varepsilon>0$, let 
$G\in \mathcal T_{\mathcal E}$ be a tree, and let $f_G\colon \mathbb T_{\mathcal E}\to G$ be an
$\varepsilon$-epimorphism. Observe that, by Theorem \ref{end-to-end-rooted-reverse}
the image $f_G(p)$ is an end vertex of $G$.
Define a tree $H$ which is the one point union of two copies $G_1$ and $G_2$ of $G$ such that $r(H)=r(G_1)=r(G_2)$. Define $p_1$ and $p_2$ as end vertices in $G_1$ and $G_2$ that correspond to the vertex $f_G(p)$. 
Let $f_G^H\colon H\to G$ be the natural identification epimorphism and let $f_H\colon \mathbb T_{\mathcal E}\to H$ be an epimorphism such that $f_G=f^H_G\circ f_H$ and 
 $f_H(p)=p_1$.
By Proposition \ref{end-to-end} there is an end vertex $q\in(f_H)^{-1}(p_2)$ and we have that both $p$ and $q$ are end vertices in $\mathbb T_{\mathcal E}$ and $d(p,q)<\varepsilon$ as needed.
\end{proof}

\begin{theorem}\label{end-vertices-closed}
The set of end vertices of the projective Fra\"{\i}ss\'e limit $\mathbb T_{\mathcal E}$ of $\mathcal T_{\mathcal E}$ is closed. 
\end{theorem}

\begin{proof}
Denote the set of end vertices of a topological tree $T$ as $EV(T)$. It follows from Theorem \ref{end-to-end-rooted-reverse} that  $EV(\mathbb{T}_{\mathcal {E}})=\bigcap\{f^{-1}(EV(T)):T\in \mathcal T_{\mathcal E} \text{ and } f\colon \mathbb{T}_{\mathcal {E}} \to T \text{ is an epimorphism}\}$ showing that the set of end vertices of $\mathbb{T}_{\mathcal {E}}$ is closed.
\end{proof}
\begin{theorem}
 The topological realization $|\mathbb T_{\mathcal E}|$ of $\mathbb{T}_{\mathcal {E}}$ is the Cantor fan.
\end{theorem}

\begin{proof}
Collecting results of this section we see that the set of end vertices of $\mathbb{T}_{\mathcal {E}}$  is closed and dense in itself, i.e. a perfect set and $\mathbb {T}_{\mathcal {E}}$ is a smooth fan since $\mathbb{T}_{\mathcal {E}}$ is isomorphic to $\mathbb{F}_{\mathcal {E}}$ and $\mathbb {F}_{\mathcal {E}}$ is a smooth fan. 

By Corollary \ref{smooth-dendroid}  $|\mathbb {T}_{\mathcal {E}}|$ is a smooth dendroid.  It follows from Theorems \ref{rooted-end} and \ref{ram-to-ram} that $\varphi(r(\mathbb {T}_{\mathcal {E}}))$, where $\varphi$ is the quotient map, is a ramification point of $|\mathbb {T}_{\mathcal {E}}|$ and from Observation \ref{arcs} that $\varphi(r(\mathbb {T}_{\mathcal {E}}))$ is the only ramification point of $|\mathbb {T}_{\mathcal {E}}|$ hence $|\mathbb {T}_{\mathcal {E}}|$ is a smooth fan.  That the set of end points of $|\mathbb {T}_{\mathcal {E}}|$ is perfect follows from the continuity of $\varphi$. The Cantor fan is characterized by being a smooth dendroid with a perfect set of end points so $|\mathbb {T}_{\mathcal {E}}|$ is a Cantor fan as claimed.

\end{proof}

\section{Confluent epimorphisms and rooted trees}
In this section we will use results obtained in the previous sections to investigate the family of rooted trees with confluent order preserving epimorphisms.

\begin{notation}
The family of all rooted trees and confluent order preserving epimorphisms is denoted by $\treecon$.
\end{notation}

\begin{theorem}\label{confluent-projective-family}
The  family  of  all  rooted trees  and  confluent epimorphisms $\treecon$ is a projective Fra\"{\i}ss\'e family.

\end{theorem}
\begin{proof}
We need to verify condition (4) of Definition \ref{definition-Fraisse}, the other conditions are straightforward.

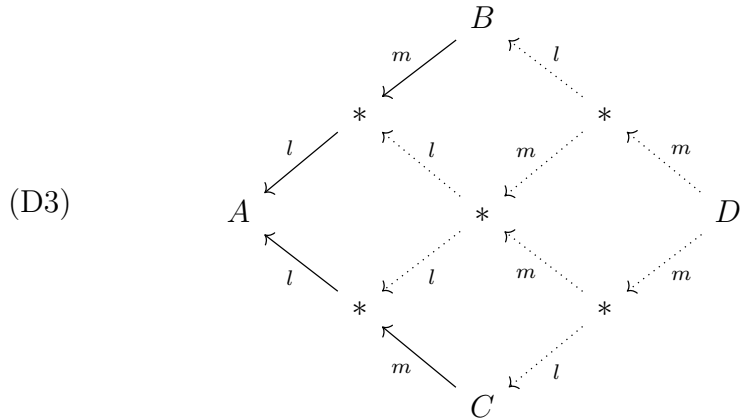
\begin{equation}\tag{D3}
\begin{tikzcd}
&&B\arrow[ld,swap,"m"]\\
&*\arrow[ld,swap,"l"]&&* \arrow[ld,swap,dotted,"m"]\arrow[lu,dotted,swap,"l"]\\
A&&*\arrow[lu,dotted,swap,"l"]\arrow[ld,dotted,"l"]&&D\arrow[lu,dotted,swap,"m"]\arrow[ld,dotted,"m"]\\
&*\arrow[lu,"l"]&&*\arrow[lu,dotted,"m"]\arrow[ld,dotted,"l"]\\
&&C\arrow[lu,"m"]
\end{tikzcd}
\end{equation}
To do so, apply the monotone-light factorization, Theorem \ref{m-l-factorization}, to functions $f$ and $g$ of diagram (D1). Then working the diagram (D3) above from left to right, closing the diagrams by using Lemma \ref{amalgamation-light} if one of the epimorphism
is light or Theorem \ref{monotone-implies-monotone} if one is monotone.   Finally we can use Theorem \ref{amalgamation-monotone} when both of the epimorphisms are monotone to construct the needed tree $D$. Note that all considered epimorphisms in the diagram (D3) are confluent by Theorems \ref{m-l-factorization-confluent} and
\ref{standard-confluent}.

\end{proof}

\begin{theorem}\label{only-rampt-or-endpt}
Let $\mathbb{T}_{\mathcal C}$ be the projective Fra\"{\i}ss\'e limit of the family $\treecon$. A vertex $r\in \mathbb{T}_{\mathcal C}$ is
either an end vertex or a ramification vertex of $\mathbb{T}_{\mathcal C}$.
\end{theorem}

\begin{proof}
Suppose a vertex $r\in \mathbb{T}_{\mathcal C}$ is not an end vertex. Then there are vertices $a,b\in V(\mathbb{T}_{\mathcal C})$ such that
$a\le r\le b$ and $a\ne r \ne b$. Let $0<\varepsilon<\min\{d(a,r),d(r,b),d(a,b)\}$, 
$Z\in \mathcal{T}_{\mathcal LC}$, and epimorphism $f_Z\colon \mathbb{T}_{\mathcal LC}\to Z$ be an $\varepsilon$-map. Define a tree $W$ by
$V(W)=V(Z)\cup \{d\}$ and $E(W)=E(Z)\cup\{\langle f_Z(r),d\rangle\}$, and an epimorphism $f_Z^W\colon W\to Z$ by $f_Z^W(x)=x$ if $x\in Z$ and $f_Z^W(d)=f_Z(r)$.
Let $f_W\colon \mathbb{T}_{\mathcal C}\to W$ be an epimorphism such that $f_Z=f_Z^W\circ f_W$ and let $C$ be the component of $(f_W)^{-1}(\{f_W(r),d\})$; by confluence of $f_W$, see Proposition~\ref{confluent-projections}, $f(C)=\{f_W(r),d\}$, so there is an arc $A$ in $\mathbb{T}_{\mathcal C}$ with $r$ as one of its vertices, such that the three arcs: $A$, $ra$, $rb$ witness that $r$ is a ramification vertex of $\mathbb{T}_{\mathcal C}$.
\end{proof}

\begin{proposition}\label{limit-dense-endpts}
Let $\mathbb{T}_{\mathcal C}$ be the projective Fra\"{\i}ss\'e limit of the family $\treecon$. Then the set of end vertices is dense in $\mathbb{T}_{\mathcal C}$.
\end{proposition}

\begin{proof}
Let $p\in \mathbb{T}_{\mathcal C} $ and let $\varepsilon >0$ be given. We want to find an end vertex $e\in\mathbb{T}_{\mathcal C}$ such that
$d(p,e)<\varepsilon$. Let $G$ be a tree and $f_G\colon \mathbb{T}_{\mathcal C}\to G$ be an
$\varepsilon$-map. Define a tree $H$ such that $V(H)=V(G)\cup \{q\}$ and $E(H)=E(G)\cup \{\langle f_G(p),q\rangle\}$, and let
$f^H_G$ be an epimorphism defined by ${f^H_G}|_{V(G)}$ is the identity and $f^H_G(q)=f_G(p)$.  Let $f_H\colon \mathbb{T}_{\mathcal C}\to H$
be an epimorphism such that $f_G^H \circ f_H=f_G$. By Proposition \ref{end-to-end} there is $e\in \mathbb{T}_{\mathcal C}$ which is an end vertex and $f_H(e)=q$.  Then
$d(p,e)<\varepsilon$ since $f_G(p)=f_G(e)$ and $f_G$ is an $\varepsilon$-map.

\end{proof}

We summarize the results of this section in the following Theorem.
\begin{theorem}\label{summerizing-confluent}
The projective Fra\"{\i}ss\'e limit $\mathbb{T}_{\mathcal C}$ of the family $\treecon$ has transitive set of edges  and its topological realization $|\mathbb{T}_{\mathcal C}|$
is a Kelley dendroid (thus smooth) with a dense set of end points and such that each point is either an end point or a ramification point in the classical sense.
\end{theorem}

\begin{proof}

Proceeding as in the proof of Theorem \ref{dense} we may show that the family $\treecon$ satisfies the hypothesis of Theorem \ref{transitive}. Hence the set of edges is transitive. Knowing that $E(\mathcal{T_C})$ is transitive we see from Corollary \ref{smooth-dendroid}, Proposition \ref{Kelley-theorem} and Observation \ref{Kelley-observation} that the topological realization $|\mathbb {T_C}|$ is Kelley.

To see that the set of end points is dense in $|\mathbb {T_C}|$ assume  there exists an open set $U$ in $|\mathbb {T_C}|$ which contains no end points.  The set $\varphi^{-1}(U)$, where $\varphi$ is the quotient map, is open in $\mathbb {T_C}$. So by Proposition \ref{limit-dense-endpts} there is an end vertex $e\in \varphi^{-1}(U)$. If follows from Theorem \ref{rooted-end} that $\mathbb T_{\mathcal C}$ has no isolated end vertices.  Then by Proposition \ref{endpts-to-endpts}, $\varphi(e)$ is an end point in $U$ contrary to our assumption.   Hence the set of end points is dense in $|\mathbb {T_C}|$.

To complete the argument we need to show that every point in $|\mathbb {T_C}|$ is either an end point or a ramification point. Note that if $p$ is a point in $|\mathbb {T_C}|$ then $\varphi^{-1}(p)$ contains at most two vertices otherwise $\mathbb {T_C}$ would not be hereditarily unicoherent. If $\varphi^{-1}(p)$ is a single vertex $q$ then, by Proposition \ref{only-rampt-or-endpt}, $q$ is an end vertex or a ramification vertex of $\mathbb {T_C}$. If $q$ is an end vertex then Proposition \ref{endpts-to-endpts} and Theorem \ref{smooth-dendroid-fan} shows that $p$ is an end point of $|\mathbb {T_C}|$.  If $q$ is a ramification vertex then  Proposition \ref{ram-to-ram}  tells us that $p$ is a ramification vertex of $|\mathbb {T_C}|$.  Finally, if $\varphi^{-1}(p)=\{q,r\}$ then if follows from Proposition \ref{rooted-end} that both of these vertices must be ramification vertices and hence, by Proposition \ref{ram-to-ram}, that $p$ is a ramification point of $|\mathbb {T_C}|$.

\end{proof}

\begin{problem}
Give a topological characterization of the dendroid $|\mathbb{T}_{\mathcal C}|$.
\end{problem}

\section{Confluent epimorphisms preserving end vertices}

In this section we consider the family $\treend$ of rooted trees and confluent, order preserving epimorphisms that also preserve end vertices of the trees.

As seen in Example \ref{rooted-standard-not-tree} above or Example \ref{big-example} below, when the standard amalgamation procedure is used on rooted trees the resulting graph may not be a tree.  Next we will describe a modification of the standard amalgamation procedure, which we call the {\it standard tree amalgamation}, that does yield a rooted tree.
\begin{definition}\label{standard-tree-amalgamation}
Suppose $A$, $B$ and $C$ are rooted trees.  If $D'$ is the graph obtained via the standard amalgamation procedure then $D'$ can be viewed as a rooted partially ordered graph where $r(D')=(r(B),r(C))$ and a partial order on $D'$ is given by $(b_1,c_1)\le (b_2,c_2)$ if $b_1 \le b_2$ and $c_1 \le c_2$.

We obtain a rooted tree $D$ derived from $D'$ by the following procedure.  Let $V(D)$ be the set of all proper chains in $D'$ (see Definition \ref{chain-def}), $\langle c_1,c_2\rangle \in E(D)$ if and only if when $c_1\subseteq c_2$ then $c_2\setminus c_1$ has at most one vertex or when $c_2\subseteq c_1$ then $c_1\setminus c_2$ has at most one vertex. The root of $D$ is  the chain containing just the one vertex $r(D')$.  Further, a partial order on $D$ is given by $c_1 \le c_2$ if $c_1 \subseteq c_2$.  Note that $D$ contains no cycles. Hence $D$ is a rooted tree.

For $c\in V(D)$ define $f_0\colon D \to B$ and $g_0\colon D \to C$ by $f_0(c)=\pi_1(\max(c))$ and $g_0(c)=\pi_2(\max(c))$. Clearly then $f\circ f_0=g\circ g_0$ so $D$ is an amalgamation of the rooted trees $B$ and $C$. One should note that if the standard amalgamation, $D'$, of a pair of rooted trees is a rooted tree then the standard tree amalgamation, $D$, of the pair is isomorphic to $D'$.

Note that $f_0$ and $g_0$ are order preserving epimorphisms.  
\end{definition}

The whole procedure can be seen in the following example.

\begin{example}\label{big-example}
Let $A$ be an edge with the two vertices $0$ and $1$, let $B$ be a triod with four vertices $a,b,c,d$, and let $C$ be an arc with three
vertices $p,q,r$; edges in all graphs are pictured below. The function $f\colon B\to A$ is defined by $f(a)=0$ and $f(b)=f(c)=f(d)=1$; similarly, $g\colon C\to A$ is defined by conditions $g(p)=0$, $g(q)=g(r)=1$. In the picture below you can find the standard amalgamation $D'$, that is not a tree
and standard tree amalgamation $D$ defined above. In the standard tree amalgamation the vertices are labeled by the end vertices of the chains, so
for example, there are three vertices labeled $(c,r)$ because there are three chains in $D'$ from the root $(a,p)$ to the vertex $(c,r)$.

Note that in this example the epimorphisms $f$ and $g$ are monotone, but the epimorphisms $f_0$ and $g_0$ are not, therefore this amalgamation procedure  could not be used in Theorem \ref{amalgamation-monotone}

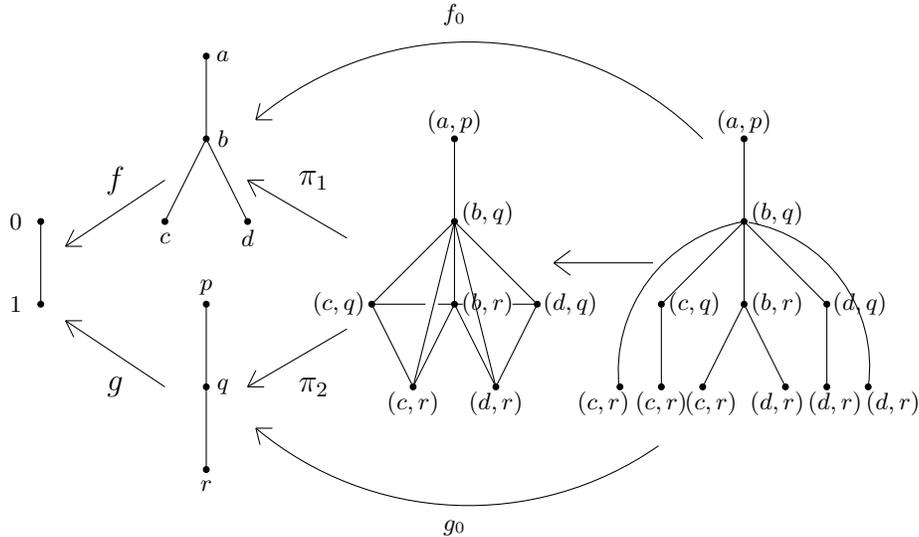
\begin{center}
\begin{tikzpicture}[scale=1.1]

  \draw (0,1) -- (0,2);
  \filldraw[black] (0,1) circle (1pt);
  \filldraw[black] (0,2) circle (1pt);
  \draw (1.5,2) -- (2,3) -- (2.5,2);
  \draw (2,3) -- (2,4);
     \filldraw[black] (1.5,2) circle (1pt);
     \filldraw[black] (2,3) circle (1pt);
    \filldraw[black] (2.5,2) circle (1pt);
  \draw (2,1) -- (2,0);
\draw (2,-1) -- (2,0);
     \filldraw[black] (2,1) circle (1pt);
    \filldraw[black] (2,0) circle (1pt);
    \filldraw[black] (2,4) circle (1pt);
    \filldraw[black] (2,-1) circle (1pt);
    \node at (-0.3,2) {\scriptsize 0};
 \node at (-0.3,1) {\scriptsize 1};
 \draw (1.5,2.5) --  (0.3,1.7);
 \draw  (0.3,1.7) --(0.4,1.9);
  \draw  (0.3,1.7) --(0.5,1.7);
  \node at (1.5,1.8) {\scriptsize $c$};
  \node at (2.5,1.8) {\scriptsize $d$};
\node at (2.2,3) {\scriptsize $b$};
\node at (2.2,4) {\scriptsize $a$};

 \draw (1.5,0) --  (0.3,0.8);
 \draw  (0.3,0.8) --(0.5,0.8);
 \draw  (0.3,0.8) --(0.4,0.6);
  \node at (0.9,0) {$g$};

\draw  (4,1) --(5,2) --(5,3);
\draw  (5,2) --(5,1);
\draw  (5,2) --(6,1);
\draw  (4,1) --(4.5,0);
\draw  (5,1) --(4.5,0);
\draw  (6,1) --(5.5,0);
\draw  (5,1) --(5.5,0);
    \filldraw[black] (4,1) circle (1pt);
    \filldraw[black] (5,2) circle (1pt);
    \filldraw[black] (5,3) circle (1pt);
    \filldraw[black] (6,1) circle (1pt);   
    \filldraw[black] (4.5,0) circle (1pt);    
    \filldraw[black] (5,1) circle (1pt);      
    \filldraw[black] (5.5,0) circle (1pt);   
\draw (4,1) --(4.65,1);
\draw (4.8,1) -- (5,1);
\draw (5,1) -- (5.1,1);
\draw (5.7,1) -- (6,1);
\draw (5,2) -- (4.5,0);
\draw (5,2) -- (5.5,0);

\node at (5,3.2) {\scriptsize $(a,p)$};
\node at (5.4,2.1) {\scriptsize $(b,q)$};
\node at (3.6,1) {\scriptsize $(c,q)$};
\node at (5.4,1) {\scriptsize $(b,r)$};
\node at (6.4,1) {\scriptsize $(d,q)$};
\node at (4.5,-0.2) {\scriptsize $(c,r)$};
\node at (5.5,-0.2) {\scriptsize $(d,r)$};

\draw  (7.5,1) --(8.5,2) --(8.5,3);
\draw  (8.5,2) --(8.5,1);
\draw  (8.5,2) --(9.5,1);
\draw  (7.5,1) --(7.5,0);
\draw  (8.5,1) --(8,0);
\draw  (9.5,1) --(9.5,0);
\draw  (8.5,1) --(9,0);

\draw (8,3) arc (45:130:4);
\draw (2.6,3.23) -- (2.8,3.23);
\draw (2.6,3.23) -- (2.65,3.43);
\node at (5,4.5) {\scriptsize $f_0$};
\draw (2.6,-0.5) arc (230:305:4);
\node at (5,-1.7) {\scriptsize $g_0$};
\draw (2.6,-0.5) -- (2.8,-0.5);
\draw (2.6,-0.5) -- (2.65,-0.7);

    \filldraw[black] (7.5,1) circle (1pt);
    \filldraw[black] (8.5,2) circle (1pt);
    \filldraw[black] (8.5,1) circle (1pt);
    \filldraw[black] (9.5,1) circle (1pt);   
    \filldraw[black] (9.5,0) circle (1pt);    
    \filldraw[black] (9,0) circle (1pt);      
    \filldraw[black] (7.5,0) circle (1pt);  
    \filldraw[black] (8,0) circle (1pt);   
    \filldraw[black] (8.5,3) circle (1pt); 
    
    \filldraw[black] (7,0) circle (1 pt);
    \filldraw[black] (10,0) circle (1pt);
    \draw (8.5,2) arc (100:187:1.84);
    \draw (10,0) arc (-8:80:1.77);

\node at (8.5,3.2) {\scriptsize $(a,p)$};
\node at (8.9,2.1) {\scriptsize $(b,q)$};
\node at (7.9,1) {\scriptsize $(c,q)$};
\node at (8.9,1) {\scriptsize $(b,r)$};
\node at (9.9,1) {\scriptsize $(d,q)$};
\node at (7.5,-0.2) {\scriptsize $(c,r)$};
\node at (8.1,-0.2) {\scriptsize $(c,r)$};
\node at (8.9,-0.2) {\scriptsize $(d,r)$};
\node at (9.6,-0.2) {\scriptsize $(d,r)$};

\node at (6.8,-0.2) {\scriptsize $(c,r)$};
\node at (10.3,-0.2) {\scriptsize $(d,r)$};

 \draw (7.4,1.5) --  (6.2,1.5);
 \draw  (6.2,1.5) --(6.4,1.65);
 \draw  (6.2,1.5) --(6.4,1.35);

 \draw (3.7,1.8) --  (2.5,2.5);
 \draw  (2.5,2.5) --(2.7,2.5);
 \draw  (2.5,2.5) --(2.6,2.3);

 \draw  (3.7,0.7) --  (2.5,0);
 \draw  (2.5,0) --(2.7,0);
  \draw  (2.5,0) --(2.6,0.2);
 \node at (3.3,0) {$\pi_2$};

 \node at (0.9,2.5) {$f$};
 \node at (3.3,2.5) {$\pi_1$};

\node at (2,1.2) {\scriptsize $p$};
\node at (2.2,0) {\scriptsize $q$};
\node at (2,-1.2) {\scriptsize $r$};

\end{tikzpicture}
\end{center}
\end{example}

\begin{proposition}\label{confluent-end-preserving}
If in the standard tree amalgamation $f$ and $g$ are confluent and preserve end vertices then $f_0$ and $g_0$ preserve end vertices.
\end{proposition}

\begin{proof}
Let $e$ be an end vertex in $D$ then $(b_1,c_1)\le (b_2,c_2)\le \dots \le (b_n,c_n)=e$ is a maximal proper chain in $D'$.  Suppose $f_0(e)=b_n$ is not an end vertex of $B$.  Then there is a vertex $b_{n+1}\in V(B)\setminus \{b_n\}$ such that $b_n\le b_{n+1}$ and $\langle b_n,b_{n+1}\rangle \in E(B)$.  If $f(b_{n+1})=f(b_n)$ then $(b_1,c_1)\le (b_2,c_2)\le \dots \le (b_n,c_n)\le (b_{n+1},c_n)$ is a chain in $D'$ containing $e$ contradicting the maximality of $e$. 

So $f$ maps the edge $\langle b_n,b_{n+1}\rangle$ to the edge $P=\langle f(b_n),f(b_{n+1})\rangle$ in $A$. Because $g$ is confluent we know, by (3) of Proposition \ref{confluent-edges}, that the component of $g^{-1}(P)$ that contains $c_n$ contains vertices $c_{n+1},\dots, c_m,c_{m+1}$ such that 
$c_n\le c_{n+1}\le\dots\le c_m\le c_{m+1}$ and $g(c_n)=g(c_{n+1})=\dots =g(c_m)$ and $g(c_{m+1})=f(b_{n+1})$
Thus we have the chain $(b_1,c_1)\le (b_2,c_2)\le \dots \le (b_n,c_n)\le (b_{n},c_{n+1})\le (b_n,c_m)\le (b_{n+1},c_{m+1})$ in $D'$ that contains $e$ again contradicting the maximality of $e$. So $f_0(e)$ is an end vertex of $B$. A similar argument shows that $g_0(e)$ is an end vertex of $C$.
\end{proof}

\begin{example}
In the standard tree amalgamation if $f$ and $g$ preserve end vertices it may not be the case that $f_0$ and $g_0$ preserve end vertices. To see this let $A$ be a triod with end vertices $1,2,3$, ramification vertex $0$, and let $1$ be the root of $A$. Let $B$ and $C$ both be arcs consisting of 5 vertices, $d,b,a,c,e$ and $s,q,p,r,t$ with the end vertices of $B$ being $d,e$, the end vertices of $C$ being $s,t$, and the roots of $B$ and $C$ being $r(B)=a$ and $r(C)=p$. Finally, define $f(a)=g(p)=1$, $f(b)=f(c)=g(q)=g(r)=0$, $f(d)=g(s)=2$, and $f(e)=g(t)=3$. So $f$ and $g$ are end vertex preserving. The standard amalgamation procedure results in $D'$ being a 4-od with the root $(a,p)$ which is also the ramification vertex of the 4-od. Since, in this case, the rooted tree $D$ obtained by the standard tree amalgamation is isomorphic to $D'$ we just consider $D'$.  Two of the branches of $D'$ are the edges $\langle (a,p),(b,r)\rangle$ and $\langle (a,p),(c,q)\rangle$. So $(b,r)$ and $(c,q)$ are end vertices in $D'$ but $f_0$ and $g_0$ applied to these vertices are not end vertices in $B$ and $C$.  Note that here $f$ and $g$ are not confluent.
\end{example}

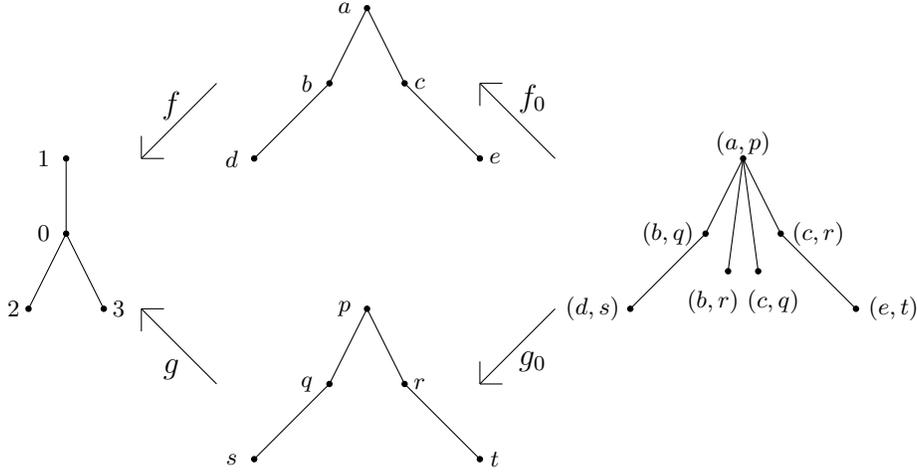
\begin{center}
\begin{tikzpicture}[scale=1]

\draw (0,2) -- (0.5,3) -- (0.5,4);
\draw (0.5,3) -- (1,2);
\filldraw[black] (0,2) circle (1pt);
\filldraw[black] (0.5,3) circle (1pt);
\filldraw[black] (0.5,4) circle (1pt);
\filldraw[black] (1,2) circle (1pt);
 \node at (-0.2,2) {\scriptsize 2};
 \node at (0.2,3) {\scriptsize 0};
 \node at (0.2,4) {\scriptsize 1};
 \node at (1.2,2) {\scriptsize 3};

\draw (3,4) -- (4,5) -- (4.5,6) -- (5,5) -- (6,4);
\filldraw[black] (3,4) circle (1pt);
\filldraw[black] (4,5) circle (1pt);
\filldraw[black] (4.5,6) circle (1pt);
\filldraw[black] (5,5) circle (1pt);
\filldraw[black] (6,4) circle (1pt);
\node at (2.7,4) {\scriptsize $d$};
 \node at (3.7,5) {\scriptsize $b$};
 \node at (4.2,6) {\scriptsize $a$};
 \node at (5.2,5) {\scriptsize $c$};
  \node at (6.2,4) {\scriptsize $e$};
  
\draw (2.5,5) -- (1.5,4);
\draw (1.5,4) -- (1.8,4);
\draw (1.5,4) -- (1.5,4.3);
\node at (1.9,4.7) {$f$};

\draw (3,0) -- (4,1) -- (4.5,2) -- (5,1) -- (6,0);
\filldraw[black] (3,0) circle (1pt);
\filldraw[black] (4,1) circle (1pt);
\filldraw[black] (4.5,2) circle (1pt);
\filldraw[black] (5,1) circle (1pt);
\filldraw[black] (6,0) circle (1pt);
\node at (2.7,0) {\scriptsize $s$};
 \node at (3.7,1) {\scriptsize $q$};
 \node at (4.2,2) {\scriptsize $p$};
 \node at (5.2,1) {\scriptsize $r$};
  \node at (6.2,0) {\scriptsize $t$};
  
\draw (2.5,1) -- (1.5,2);
\draw (1.5,2) -- (1.8,2);
\draw (1.5,2) -- (1.5,1.7);
\node at (1.9,1.2) {$g$};
  
\draw (8,2) -- (9,3) -- (9.5,4) -- (10,3) -- (11,2);
\draw (9.3,2.5) -- (9.5,4) -- (9.7,2.5);
\filldraw[black] (8,2) circle (1pt);
\filldraw[black] (9,3) circle (1pt);
\filldraw[black] (9.5,4) circle (1pt);
\filldraw[black] (10,3) circle (1pt);
\filldraw[black] (11,2) circle (1pt);
\filldraw[black] (9.3,2.5) circle (1pt);
\filldraw[black] (9.7,2.5) circle (1pt);
 \node at (7.5,2) {\scriptsize $(d,s)$};
 \node at (8.5,3) {\scriptsize $(b,q)$};
 \node at (9.5,4.2) {\scriptsize $(a,p)$};
 \node at (10.5,3) {\scriptsize $(c,r)$};
 \node at (11.5,2) {\scriptsize $(e,t)$};
 \node at (9.1,2.1) {\scriptsize $(b,r)$};
 \node at (9.9,2.1) {\scriptsize $(c,q)$};
 
 \draw (7.0,2) -- (6,1);
\draw (6,1) -- (6.3,1);
\draw (6,1) -- (6,1.3);
\node at (6.7,1.3) {$g_0$};

\draw (7.0,4) -- (6,5);
\draw (6,5) -- (6.3,5);
\draw (6,5) -- (6,4.7);
\node at (6.7,4.8) {$f_0$};

\end{tikzpicture}
\end{center}

\begin{proposition}\label{standard-tree-confluent}
Using the standard tree amalgamation procedure and the notation of diagram {\rm (D1)}, if the epimorphism $f$ is confluent,  then the epimorphism $g_0$ is confluent.
\end{proposition}

\begin{proof}
We will use the characterization in Proposition \ref{confluent-edges} (2) to show that $g_0$ is confluent. 
Let $\langle c,c'\rangle\in E(C)$ and let 
$d=\{(b_1,c_1), \dots , (b_n,c_n)\}$ be a vertex in $D$ that such that $g_0(d)=c$, i.e. $c_n=c$. We need to find vertices $d', d''\in V(D)$ such that 
$g_0(d')= c$, $g_0(d'')=c'$,  $\langle d', d''\rangle\in E(D)$, and $d,d'$ are in the same component of $(g_0)^{-1}(c)$. 
Because $f$ is confluent, by Proposition \ref{confluent-edges} (2) 
there is an edge $\langle b',b'' \rangle \in E(B)$ such that 
\begin{enumerate}
    \item $f(b')=f(b_n)=g(c)$;
    \item $f(b'')= g(c')$;
    \item $b',b_n$ are in the same component of $f^{-1}(f(b_n))$.
\end{enumerate}
Let $b_{n+1}, \dots b_m\in V(B)$ be such that $b_m=b'$ and 
$\langle b_i, b_{i+1}\rangle\in E(B)$ for $i\in\{n,\dots m-1\}$.
We may define

\begin{align*}
d'=&(b_1,c_1),\dots, (b_n,c_n), (b_{n+1},c_n),
\dots, (b_m, c_n),  \\
d''=&(b_1,c_1),\dots, (b_n,c_n), (b_{n+1},c_n),
\dots, (b_m, c_n), (b'',c')
\end{align*}
and verify that such defined chains, points of $D$, satisfy all the required conditions. 

\end{proof}

The following theorem follows from Propositions \ref{confluent-end-preserving} and \ref{standard-tree-confluent}.

\begin{theorem}\label{Thm-confluent-end-preserving}
The family $\treend$ is a projective Fra\"{\i}ss\'e family.
\end{theorem}

The following Proposition  is analogous to Proposition \ref{only-rampt-or-endpt}, but this time we need to have light epimorphisms.

\begin{proposition}\label{only-rampt-or-endpt-closed}

Let $\mathbb{T}_{\mathcal {CE}}$ be the projective Fra\"{\i}ss\'e limit of the family $\treend$. A vertex $r\in \mathbb{T}_{\mathcal {CE}}$ is
either an end vertex or a ramification vertex of $\mathbb{T}_{\mathcal {CE}}$.
\end{proposition}
\begin{proof}
Suppose a vertex $r\in \mathbb{T}_{\mathcal {CE}}$ is not an end vertex. Then the set $P=\{x\in\mathbb{T}_{\mathcal {CE}}:r\le x\}$ is a connected nondegenerate subtree of $\mathbb{T}_{\mathcal {CE}}$.
Let $0<\varepsilon<\diam (P)$, and let an epimorphism $f_Z\colon \mathbb{T}_{\mathcal {CE}}\to Z$ be an $\varepsilon$-map thus $f_Z(P)$ is nondegenerate.

Define a tree $W$ by $V(W)=V(Z)\times \{1,2,3\}/\sim$, where $(x,i)\sim (y,j)$ if and only if $x=y$ and $x\not\in P$ or $x=f_Z(r)$ and $\langle (x,i),(y,i)\rangle\in E(W)$ if and only if $\langle x,y\rangle\in E(Z)$, with $f^W_Z\colon W\to Z$ being just the projection. In other words, $W$ is the graph $Z$ with three copies of $f_Z(P)$ attached at $f_Z(r)$.

By \cite[Theorem 2.4]{Pseudo} there is a sequence $T_1=W, T_2, \dots$ of members of $\mathcal T_{\mathcal {CE}}$ and epimorphisms $f_1,f_2,\dots$ with
$f_i\colon T_{i+1}\to T_i$ such that the inverse limit space $\iLim \{T_i,f_i\}$ is isomorphic to $\mathbb{T}_{\mathcal {CE}}$. We define, by induction, a sequence
$A^j_1,A^j_2,\dots$, $j\in \{1,2,3\}$ of subtrees of $T_1, T_2,\dots$, respectively such that the following holds for $j\in \{1,2,3\}$:

\begin{enumerate}
    \item $A^j_1=P\times \{j\}$;
    \item $f_i(A^j_{i+1})=A^j_i$;
    \item $f_{T_i}\colon \mathbb T_{\mathcal {CE}}\to T_i$ is an epimorphism such that $f_{T_i}=f_i\circ f_{T_{i+1}}$;
    \item $A^j_i\cap A^k_i=\{f_{T_i}(r)\} $ for $k\ne j$;
\end{enumerate}
Assume that $\{A^j_1,\dots, A^j_n\}$ satisfy these conditions.
Then, since the epimorphisms in $\treend $ are confluent,  define $A^j_{i+1}$ as components of $(f_i)^{-1}(A^j_i)$ containing the point $f_{T_{i+1}}(r)$. One can verify that conditions (1)-(4) are satisfied.
Then the trees $A^j=\iLim \{A^j_i, f_i|_{A^j_{i+1}}\}$ satisfy $A^j\cap A^k=\{r\}$, for $j\ne k$, so $r$ is a ramification vertex of $\mathbb T_{\mathcal {CE}}$.
\end{proof}

\begin{theorem}\label{Thm-CE}
The projective Fra\"{\i}ss\'e limit of $\treend$ has transitive set of edges  and the topological realization $|\mathbb{T}_{\mathcal E}|$ of the limit $\mathbb{T}_{\mathcal {CE}}$ is a Kelley dendroid (thus smooth) with a closed set of end points and such that each point is either an end point or a ramification point in the classical sense.
\end{theorem}
\begin{proof}

The results used in Theorem \ref{summerizing-confluent} to show that $\mathbb T_{\mathcal {C}}$ has a transitive set of edges and its topological realization is a Kelley dendroid only required that the family  $\mathbb T_{\mathcal C}$ consisted of rooted trees with confluent order-preserving epimorphisms.  Since this is also true for $\mathbb T_{\mathcal {CE}}$ we see that $\mathbb T_{\mathcal {CE}}$ has a transitive set of edges and $|\mathbb T_{\mathcal {CE}}|$ is a Kelley dendroid.

Theorem \ref{end-vertices-closed} shows that the set of end vertices of $\mathbb T_{\mathcal {CE}}$ is closed. Since $\mathbb{T}_{\mathcal {CE}}$ is compact and $\varphi$ is continuous it follows that $\varphi(EP(\mathbb{T}_{\mathcal {CE}}))$ is closed.  Using an argument similar to that in Theorem \ref{summerizing-confluent} we can see that only end vertices map to end points under $\varphi$ so $ EP(|\mathbb{T}_{\mathcal {CE}}|)=\varphi(EP(\mathbb{T}_{\mathcal {CE}}))$.

Finally, that $|\mathbb T_{\mathcal {CE}}|$ contains only end points or ramification points follows from Theorem \ref{only-rampt-or-endpt-closed} in the same manner as in Theorem \ref{summerizing-confluent}.
\end{proof}

The class of dendroids known as Mohler-Nikiel dendroids, first constructed in \cite{Dendroid}, 
have all of the properties $|\mathbb T_{\mathcal {CE}}|$ was just shown to possess. The authors do not know if  $|\mathbb T_{\mathcal {CE}}|$ is homeomorphic to a Mohler-Nikiel dendroid.  This leads to the following problem.

\begin{problem}
Give a topological characterization of the dendroid $|\mathbb{T}_{\mathcal {CE}}|$.
\end{problem}

We are currently working on two more articles in the subject: one on homogeneity properties of Fra\"{\i}ss\'e limits and a second on the Fra\"{\i}ss\'e family of all graphs with confluent epimorphisms.


\end{document}